\documentclass[journal, letterpaper]{IEEEtran}
\usepackage{amsmath,amsfonts,amssymb,euscript, graphicx, 
epsfig,
enumerate,float,afterpage, subfigure, ifthen}%

\newtheorem{cor}{Corollary}
\newtheorem{lem}{Lemma}

\newcommand{\expect}[1]{\mathbb{E}\left[#1\right]}
\newcommand{\defequiv}{\mbox{\raisebox{-.3ex}{$\overset{\vartriangle}{=}$}}}

\newcommand{\norm}[1]{||{#1}||}

\newcommand{\script}[1]{{{\cal{#1} }}}

\begin{document}

\title
  {Energy-Aware Wireless Scheduling with Near Optimal Backlog and Convergence Time Tradeoffs}
\author{Michael J. Neely\\University of Southern California\\$\vspace{-.5in}$
\thanks{The author is with the  Electrical Engineering department at the University
of Southern California, Los Angeles, CA.} 
\thanks{This work is supported in part  by one or more of:  the NSF Career grant CCF-0747525,  the 
Network Science Collaborative Technology Alliance sponsored
by the U.S. Army Research Laboratory W911NF-09-2-0053.}
}

\markboth{}{Neely}

\maketitle

\begin{abstract}   
This paper considers a wireless link with randomly arriving data that is queued and 
served over a time-varying channel.  It is known that any algorithm that 
comes within $\epsilon$ of the minimum average power required for queue stability must
incur average queue size at least $\Omega(\log(1/\epsilon))$.    However, the optimal convergence
time is unknown, and prior algorithms give convergence time bounds of  $O(1/\epsilon^2)$. 
This paper develops a scheduling algorithm that, for any $\epsilon>0$,  achieves the optimal $O(\log(1/\epsilon))$ 
average queue size 
tradeoff with an improved convergence time of $O(\log(1/\epsilon)/\epsilon)$.  This is shown to be within a logarithmic
factor of the best possible convergence time.   The method uses the simple drift-plus-penalty technique with an improved convergence time analysis. 
\end{abstract} 

\section{Introduction}

This paper considers power-aware scheduling in a wireless link with a time-varying channel
and randomly arriving data.  Arriving data is queued for eventual transmission.  The transmission rate
out of the queue is determined by the current channel state and the current power allocation decision.  
Specifically, the controller can make an \emph{opportunistic scheduling} decision by observing the channel before allocating power.  For a given $\epsilon>0$, the goal is to push average power to within $\epsilon$ of the minimum possible average power required for queue stability while ensuring optimal queue size and convergence time tradeoffs.  

A major difficulty is that the data arrival rate and the channel probabilities are unknown. Hence, the  convergence time of an algorithm includes the \emph{learning time} associated with estimating probability distributions or ``sufficient statistics'' of these distributions.   The optimal learning time required to achieve the average 
power and backlog objectives, as well as the appropriate sufficient statistics to learn,  are unknown.   This open question is important because it determines how fast an algorithm can adapt to its environment. 
A contribution of the current paper is the development of 
an algorithm that, under suitable assumptions, provides an optimal power-backlog tradeoff while 
provably coming within a logarithmic factor of the optimal convergence time.   This is done via the existing 
drift-plus-penalty algorithm but with an improved convergence time analysis. 

Work on opportunistic scheduling was pioneered by Tassiulas and Ephremides in \cite{tass-server-allocation}, where the Lyapunov method and the \emph{max-weight} algorithms were introduced for queue stability.  
Related opportunistic scheduling work that focuses on utility optimization 
is given in \cite{atilla-fairness-ton}\cite{atilla-primal-dual-jsac}\cite{lee-stochastic-scheduling}\cite{prop-fair-down}\cite{vijay-allerton02}\cite{stolyar-greedy}\cite{neely-fairness-ton}\cite{sno-text} using dual, primal-dual, and stochastic gradient methods, and in \cite{shroff-opportunistic} using index policies. 
The basic drift-plus-penalty algorithm of Lyapunov optimization can be viewed as a \emph{dual method}, and 
is known to provide, for any $\epsilon>0$, 
an $\epsilon$-approximation to minimum average power with a corresponding $O(1/\epsilon)$ tradeoff in average queue size \cite{sno-text}\cite{neely-energy-it}. This tradeoff is not optimal.    Work by Berry and Gallager in \cite{berry-fading-delay} shows that, for queues with \emph{strictly concave} rate-power curves, any algorithm that achieves an $\epsilon$-approximation must incur average backlog of $\Omega(\sqrt{1/\epsilon})$, even if that algorithm knows all system probabilities.  Work in \cite{neely-energy-delay-it} shows this tradeoff is achievable (to within a logarithmic factor) using an algorithm that does \emph{not} know the system probabilities.  The work \cite{neely-energy-delay-it} further considers the exceptional case when rate-power curves are \emph{piecewise linear}.  In that case, an improved tradeoff of $O(\log(1/\epsilon))$ is both achievable and  optimal.  This is done using an exponential Lyapunov function together with a drift-steering argument.  Work in \cite{longbo-LIFO-ton}\cite{longbo-lagrange-tac} shows that similar logarithmic tradeoffs are possible via the basic drift-plus-penalty algorithm with  Last-in-First-Out scheduling. 

Now consider the question of \emph{convergence time}, being the time required for the average queue size and power guarantees to kick in. This convergence time question is unique to problems of  \emph{stochastic scheduling when system probabilities are unknown}.  If probabilities were known, the optimal fractions of time for making certain decisions could be computed offline (possibly via a very complex optimization), so that system averages would ``kick in'' immediately at time 0.  Thus, convergence time in the context of this paper 
should not be confused with \emph{algorithmic complexity} for non-stochastic 
optimization problems. 

Unfortunately, prior work that treats stochastic scheduling with unknown probabilities, 
including the basic drift-plus-penalty algorithm as well as extensions 
that achieve square root and logarithmic tradeoffs, 
give  only $O(1/\epsilon^2)$ convergence time guarantees.    Recent work in \cite{atilla-convergence-infocom2013} treats 
convergence time for a related problem of flow rate allocation and concludes that constraint violations decay as $c(\epsilon)/t$, where $c(\epsilon)$ is a constant that depends on $\epsilon$ and $t$ is the total time the algorithm has been in operation. While the work \cite{atilla-convergence-infocom2013} does not specify the size of the 
$c(\epsilon)$ constant, it can be shown that $c(\epsilon) = O(1/\epsilon)$.  Intuitively, this is because the $c(\epsilon)$ value is related to an average queue size, which is $O(1/\epsilon)$. The time $t$ needed to ensure
constraint violations are at most $\epsilon$ is found by solving $c(\epsilon)/t = \epsilon$. The simple answer is $t = O(1/\epsilon^2)$, again exhibiting $O(1/\epsilon^2)$ convergence time!  
This leads one to suspect that $O(1/\epsilon^2)$ is \emph{optimal}.   

This paper shows, for the first time, that $O(1/\epsilon^2)$ convergence time is \emph{not optimal}.  Specifically, under the same piecewise linear assumption in \cite{neely-energy-delay-it}, and for the special case of a system with just one queue, it is shown that the existing drift-plus-penalty algorithm yields an $\epsilon$-approximation  with both $O(\log(1/\epsilon))$ average queue size and $O(\log(1/\epsilon)/\epsilon)$ convergence time.  This is an encouraging result that shows learning times for power-aware scheduling can be pushed much smaller than expected.  


The next section specifies the problem formulation.  Section \ref{section:converse} shows a lower bound on 
convergence time of $\Omega(1/\epsilon)$.  Section \ref{section:algorithm} develops an algorithm that achieves this bound to within a logarithmic factor. 

\section{System model} \label{section:model} 

Consider a wireless link with randomly arriving traffic.  The system operates in slotted time with slots $t \in \{0, 1, 2, \ldots\}$.  Data arrives every slot and is queued for transmission.  Define:  
\begin{eqnarray*}
Q(t) &=& \mbox{queue backlog on slot $t$} \\
a(t) &=& \mbox{new arrivals on slot $t$} \\
\mu(t) &=& \mbox{service offered on slot $t$} 
\end{eqnarray*}
The values of $Q(t), a(t), \mu(t)$ are nonnegative and their units 
depend on the system of interest.  For example, they 
can take integer units of \emph{packets} (assuming packets have fixed size), or 
real units of \emph{bits}.   Assume the queue is initially empty, so that $Q(0)=0$. 
The queue dynamics  are: 
\begin{equation} \label{eq:q-update} 
 Q(t+1) = \max[Q(t) + a(t) - \mu(t), 0] 
 \end{equation} 
 
 Assume that $\{a(t)\}_{t=0}^{\infty}$ is an independent and identically distributed (i.i.d.) sequence with mean $\lambda = \expect{a(t)}$.  For simplicity, assume the amount of arrivals in one slot is bounded by a constant $a_{max}$, so that 
 $0 \leq a(t) \leq a_{max}$ for all slots $t$. 
 
 If the controller
 decides to transmit data on slot $t$, it uses one unit of power.  Let $p(t) \in \{0,1\}$ be the power used on slot $t$. 
 The amount of data that can be transmitted depends on the current channel state.  Let $\omega(t)$ be the amount
 of data that can be transmitted on slot $t$ if power is allocated, so that: 
 \[ \mu(t) = p(t)\omega(t) \]
 Assume that $\omega(t)$ is i.i.d. over slots and takes values in a finite set $\Omega = \{\omega_0, \omega_1, \omega_2, \ldots, \omega_M\}$, where $\omega_0=0$ and $\omega_i$ is a positive real number for 
 all $i \in \{1, \ldots, M\}$.  Assume these values are ordered so that: 
 \[ 0 = \omega_0 < \omega_1 < \omega_2 < \cdots < \omega_M\]
 For each $\omega_k \in \Omega$, define $\pi(\omega_k) = Pr[\omega(t)=\omega_k]$.  
 
 Every slot $t$ the system controller observes $\omega(t)$ and then chooses $p(t) \in \{0,1\}$.  The choice $p(t)=1$ activates the link for transmission of $\omega(t)$ units of data.  Fewer than $\omega(t)$ units are transmitted if $Q(t)<\mu(t)$ (see the queue equation \eqref{eq:q-update}).   The largest possible average transmission rate is $\expect{\omega(t)}$, which is achieved by using $p(t)=1$ for all $t$.  It is assumed throughout that $0 \leq \lambda \leq \expect{\omega(t)}$.

\subsection{Optimization goal}

For a real-valued random process $b(\tau)$ that evolves over slots $\tau \in \{0, 1, 2, \ldots\}$, 
define its time average expectation over $t>0$ slots as: 
\begin{equation} \label{eq:overline-def} 
\overline{b}(t) \defequiv \frac{1}{t}\sum_{\tau=0}^{t-1} \expect{b(\tau)}
\end{equation} 
where ``$\defequiv$'' represents ``defined to be equal to.'' 
With this notation, $\overline{\mu}(t)$, $\overline{p}(t)$, $\overline{Q}(t)$ respectively denote the time average expected transmission rate, power, and queue size over the first $t$ slots. 

The basic stochastic optimization problem of interest is: 
\begin{eqnarray}
\mbox{Minimize:} & \limsup_{t\rightarrow\infty} \overline{p}(t) \label{eq:p1}  \\
\mbox{Subject to:} & \liminf_{t\rightarrow\infty} \overline{\mu}(t) \geq \lambda  \label{eq:p2} \\
& p(t) \in \{0,1\} \: \: \: \:  \forall t \in \{0, 1, 2, \ldots\} \label{eq:p3} 
\end{eqnarray}
The assumption $\lambda \leq \expect{\omega(t)}$ ensures the above 
problem is always \emph{feasible}, so that it is possible to satisfy constraints \eqref{eq:p2}-\eqref{eq:p3} using  
$p(t)=1$ for all $t$. 
Define $p^*$ as the infimum average power for the above problem.  
An algorithm is said to produce an \emph{$\epsilon$-approximation} at time $t$ if, for a given $\epsilon\geq0$: 
\begin{eqnarray*}
\overline{p}(t) &\leq& p^* + \epsilon \\
\lambda - \overline{\mu}(t) &\leq& \epsilon 
\end{eqnarray*}
An algorithm is said to produce an \emph{$O(\epsilon)$-approximation} if the $\epsilon$ symbols on the right-hand-side of the above two inequalities are replaced by some constant multiples of $\epsilon$. 

Fix $\epsilon>0$.  This paper shows that a simple drift-plus-penalty algorithm that takes $\epsilon$ as an input parameter (and that has no knowledge of the arrival rate or channel probabilities) can be used to ensure there is a time $T_{\epsilon}$, called the \emph{convergence time}, for which: 
\begin{itemize} 
\item The algorithm produces an $O(\epsilon)$-approximation for all $t \geq T_{\epsilon}$. 
\item The algorithm ensures the following for all $t \in \{0, 1,2, \ldots\}$: 
\begin{equation} \label{eq:bound} 
 \overline{Q}(t) \leq O(\log(1/\epsilon))
 \end{equation} 
\item $T_{\epsilon} = O(\log(1/\epsilon)/\epsilon)$. 
\end{itemize} 

The average queue size bound \eqref{eq:bound} is known to be optimal, in the sense that no algorithm 
can provide a sub-logarithmic guarantee \cite{neely-energy-delay-it}.  The next section shows that the convergence time $O(\log(1/\epsilon)/\epsilon)$ is within a logarithmic factor of the optimal convergence time. 

\section{A lower bound on convergence time} \label{section:converse} 

\subsection{Intuition} 

One type of power allocation policy is an \emph{$\omega$-only policy} that, every slot $t$, observes
$\omega(t)$ and independently chooses $p(t) \in \{0,1\}$ according to some stationary conditional probabilities 
$Pr[p(t)=1|\omega(t)=\omega]$ that are specified for all $\omega \in \Omega$.  The resulting average power and transmission rate is: 
\begin{eqnarray*}
\expect{p(t)} &=& \mbox{$\sum_{k=0}^M\pi(\omega_k)Pr[p(t)=1|\omega(t)=\omega_k]$} \\
\expect{\mu(t)} &=& \mbox{$\sum_{k=0}^M \pi(\omega_k)\omega_k Pr[p(t)=1|\omega(t)=\omega_k]$}
\end{eqnarray*}
It is known that the problem \eqref{eq:p1}-\eqref{eq:p3} is solvable over the class of $\omega$-only policies \cite{sno-text}.  Specifically, 
if the arrival rate $\lambda$ and the channel probabilities $\pi(\omega_k)$ were known in advance, 
one could offline compute an $\omega$-only policy to satisfy: 
\begin{eqnarray}
\expect{p(t)} &=& p^* \label{eq:ideal1}\\
\expect{\mu(t)} &=& \lambda \label{eq:ideal2} 
\end{eqnarray}
This is a $0$-approximation for all $t \geq 0$.  However, such an algorithm would typically incur infinite average queue size (since the service rate  equals the arrival rate).  Further, it is not possible to implement this algorithm without perfect knowledge of $\lambda$ and $\pi(\omega_k)$ for all $\omega_k \in \Omega$. 

Suppose one temporarily allows for infinite average queue size.
Consider the following thought experiment (similar to that considered for utility optimal flow allocation  
in \cite{atilla-convergence-infocom2013}). Consider an algorithm that does not know the system probabilities and hence makes a \emph{single mistake} at time $0$, so that: 
\[ \expect{p(0)} = p^* + c \]
where $c>0$ is some constant gap away from the optimal average power $p^*$.  However, suppose a genie gives the controller perfect knowledge of the system probabilities at time $1$, and then for slots $t \geq 1$ the network makes decisions to achieve the ideal averages \eqref{eq:ideal1}-\eqref{eq:ideal2}. The resulting time average expected power over the first $t>1$ slots is: 
\[ \overline{p}(t) = \frac{p^* + c}{t} + \frac{(t-1)p^*}{t} = p^* + \frac{c}{t} \]
Thus, to reach an $\epsilon$-approximation, this genie-aided algorithm requires a convergence time  
$t = c/\epsilon = \Theta(1/\epsilon)$. 

\subsection{An example with $\Omega(1/\epsilon)$ convergence time}\label{section:converseb} 

The above thought experiment does not prove an $\Omega(1/\epsilon)$ bound on convergence time because it assumes the algorithm makes decisions according to \eqref{eq:ideal1}-\eqref{eq:ideal2} for all slots $t\geq1$, which may not be the optimal way to compensate for the mistake on slot $0$.  This section defines a simple system for which convergence time is at least $\Omega(1/\epsilon)$ under any algorithm. 

Consider a system  with deterministic arrivals of 1 packet every slot (so $\lambda =1$). 
There are three possible channel states $\omega(t) \in  \{1, 2, 3\}$, with probabilities: 
\[  \pi(3)=y, \pi(2)=z, \pi(1)=1-y-z \]
For each slot $t>0$, define the \emph{system history} $\script{H}(t) = \{(a(0), \omega(0), p(0)), \ldots, (a(t-1), \omega(t-1), p(t-1))\}$. Define $\script{H}(0)=0$.
For each slot $t$, a general algorithm has conditional probabilities $\theta_i(t)$ defined for $i \in \{1, 2, 3\}$ by: 
\[ \theta_i(t) = Pr[p(t)=1|\omega(t)=i, a(t), \script{H}(t)] \] 
On a single slot, it is not difficult to show that the minimum average power $\expect{p(t)}$ required to achieve a 
given average service rate $\mu= \expect{\mu(t)}$ is characterized by the following function $h(\mu)$: 
\[ h(\mu) \defequiv  \left\{\begin{array}{cc}
\mu/3 & \mbox{ if $0 \leq \mu \leq 3y$} \\
y + (\mu-3y)/2 & \mbox{ if $3y \leq \mu \leq 3y + 2z$} \\
\mu-2y-z & \mbox{ if $3y+2z \leq \mu \leq 2y + z + 1$} 
\end{array}\right. \]
There are two significant vertex points $(\mu, h(\mu))$ for this function.  The first is $(3y,y)$, achieved by allocating 
power if and only if $\omega(t) =3$.  The second is $(3y+2z, y+z)$, 
achieved by allocating power if and only if $\omega(t) \in \{2, 3\}$.  

Define $\script{R}$ as the set of points $(\mu,p)$ that lie on or above the curve $h(\mu)$: 
\[ \script{R} = \{(\mu,p) \in \mathbb{R}^2 | 0 \leq \mu \leq 2y+z+1, h(\mu) \leq p \leq 1 \} \]
The set $\script{R}$ is convex. 
Under any algorithm one has: 
\[ (\expect{\mu(\tau)}, \expect{p(\tau)}) \in \script{R} \: \: \: \: \forall \tau \in \{0, 1,2 , \ldots\} \]
For a given $t>1$, the following two vectors must be in $\script{R}$: 
\begin{eqnarray*}
(\mu_0, p_0) &\defequiv& (\expect{\mu(0)}, \expect{p(0)}) \\
(\mu_1, p_1) &\defequiv& \frac{1}{t-1}\sum_{\tau=1}^{t-1} (\expect{\mu(\tau)}, \expect{p(\tau)}) 
\end{eqnarray*}
That $(\mu_1, p_1)$ is in $\script{R}$ follows because it is the average of points in $\script{R}$, and $\script{R}$ is convex. 
By definition of $(\overline{\mu}(t), \overline{p}(t))$: 
\begin{equation} \label{eq:case0} 
 (\overline{\mu}(t), \overline{p}(t)) = \frac{1}{t}(\mu_0, p_0)  + \frac{1-t}{t}(\mu_1, p_1) 
 \end{equation} 
 
Fix $\epsilon$ such that $0 < \epsilon < 1/64$. The algorithm must ensure $(\overline{\mu}(t), \overline{p}(t))$ is an $\epsilon$-approximation to the target point $(1, h(1))$, so that: 
\[ \overline{\mu}(t) \geq 1 - \epsilon \: \: , \: \: \overline{p}(t) \leq h(1) + \epsilon \]
The algorithm has no knowledge of the probabilities $y$ and $z$ at time $0$, so 
$\theta_1(0)$, $\theta_2(0)$, $\theta_3(0)$ are arbitrary.  Suppose a genie reveals $y$ and $z$ on slot $1$, and the network makes decisions on slots $\{1, \ldots, t-1\}$ that result in a $(\mu_1, p_1)$ vector that \emph{optimally compensates for any mistake on slot $0$}.  Thus, $(\mu_1, p_1)$ is assumed to be the vector in $\script{R}$ that ensures \eqref{eq:case0} produces an $\epsilon$-approximation in the smallest time $t$.

The following proof considers two cases:  The first case assumes $\theta_2(0)\leq 1/2$, but considers probabilities $y$ and $z$ for which 
minimizing average power requires \emph{always transmitting when $\omega(t)=2$}.  The second case assumes $\theta_2(0)>1/2$, but then considers probabilities $y$ and $z$ for which minimizing average power requires \emph{never transmitting when $\omega(t)=2$}.   In both cases, the nonlinear structure of the $h(\mu)$ curve prevents a fast recovery from the initial mistake.

\begin{itemize} 
\item Case 1: Suppose $\theta_2(0) \leq 1/2$.  Consider $y=0, z=1/4$. 
 Then $\omega(t) \in \{1,2\}$ for all $t$, $\pi(1)=3/4, \pi(2)=1/4$, and $\omega(t)=2$ is the most efficient state.   
 The $h(\mu)$ curve is shown in 
 Fig. \ref{fig:case1}.  The minimum average power to support $\lambda=1$ is $h(1)=3/4$, and so the target point is $X = (1, 3/4)$. 
 The point $(\mu_0, p_0) = (\expect{\mu(0)}, \expect{p(0)})$  is: 
 \[ (\mu_0,p_0) = \left(\frac{\theta_2(0)}{2} + \frac{3\theta_1(0)}{4}, \frac{\theta_2(0)}{4} + \frac{3\theta_1(0)}{4}\right) \]
The set of possible $(\mu_0, p_0)$ is formed by considering all  $\theta_2(0)\in[0,1/2]$, $\theta_1(0)\in[0,1]$.
This set lies inside the left (orange) shaded region of Fig. \ref{fig:case1}.  To see this, 
note that if $\theta_2(0)$ is fixed at a certain value, the resulting $(\mu_0,p_0)$ point lies on a line segment 
of slope $1$ that is formed by sweeping $\theta_1(0)$ through the interval $[0,1]$.  If $\theta_2(0)=1/2$, that line segment  is
between points $(1/4,1/8)$ and $(1,7/8)$ in Fig. \ref{fig:case1}.  If $\theta_2(0)<1/2$ then the line segment is shifted to the left.   

\begin{figure}[t]
   \centering
   \includegraphics[width=3in]{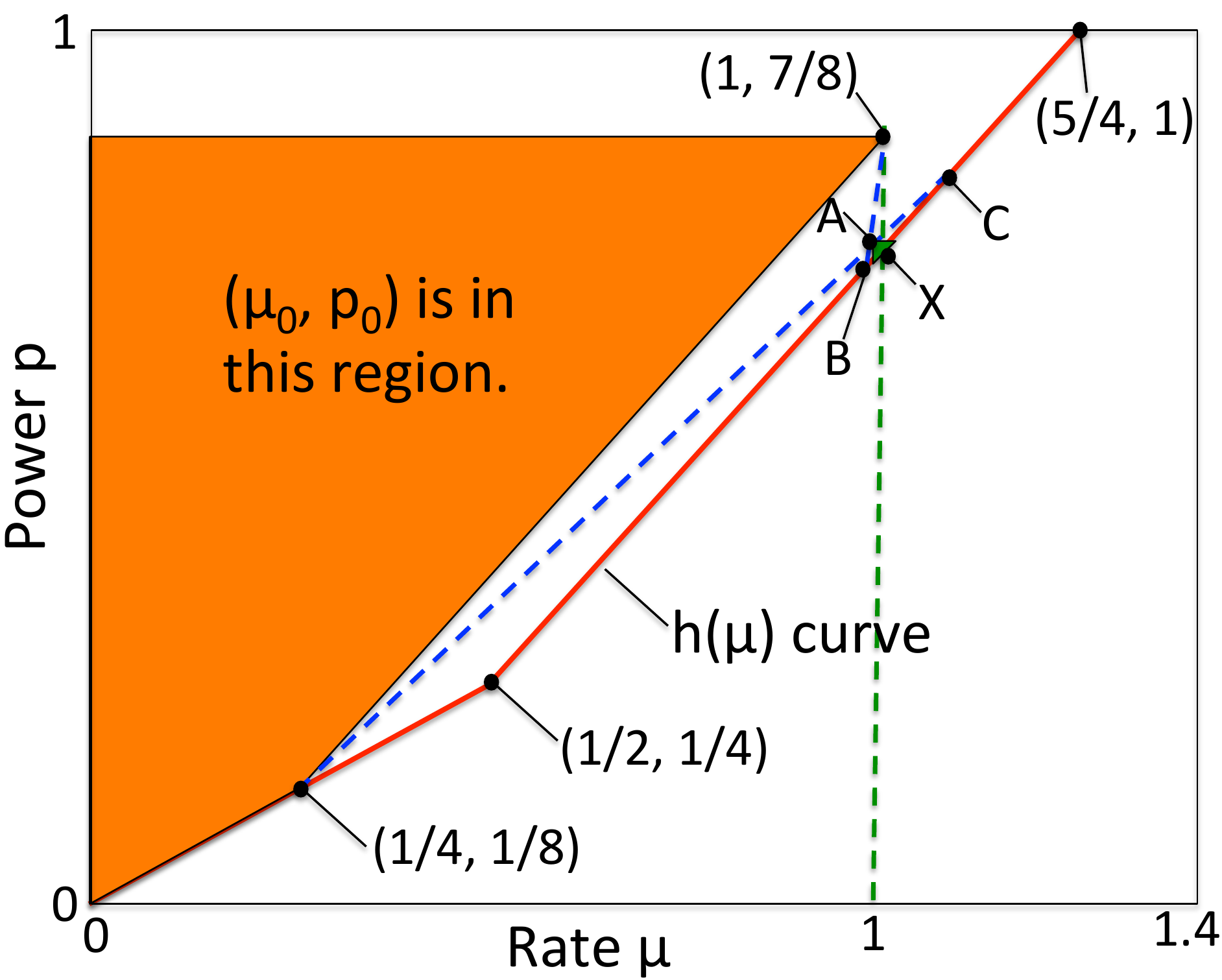} 
   \caption{The performance region for case 1.  The line segments between $(1/4, 1/8)$ and $C$ and between $(1, 7/8)$ and $B$ intersect at point $A$.}
   \label{fig:case1}
\end{figure}

The small triangular (green) shaded region in Fig. \ref{fig:case1}, with one vertex at point $A$, is the 
\emph{target region}. The vector $(\overline{\mu}(t), \overline{p}(t))$ must be in this region to be an 
$\epsilon$-approximation.   The point $A$ is defined: 
\[ A = X + (-\epsilon, \epsilon)  = (1-\epsilon, 3/4+\epsilon) \]
It suffices to search for an optimal compensation vector $(\mu_1, p_1)$ on the curve $(\mu, h(\mu))$. 
This is because 
the average power $p_1$ from a  
point $(\mu_1, p_1)$ \emph{above} the curve $(\mu, h(\mu))$ 
can be reduced, without affecting $\mu_1$, by choosing a point \emph{on} the curve. 
By geometry, $(\mu_1, p_1)$ must lie on the line segment between points $B$ and $C$ in Fig. \ref{fig:case1}, where: 
\begin{eqnarray*}
B &=& X - \left(\frac{\epsilon}{1-16\epsilon}, \frac{\epsilon}{1-16\epsilon}\right)  \\
C &=& X + \left(\frac{11\epsilon}{1-16\epsilon},  \frac{11\epsilon}{1-16\epsilon} \right) 
\end{eqnarray*}  
Indeed, if $(\mu_1, p_1)$ were on the $(\mu,h(\mu))$ curve but  \emph{not} in between points 
$B$ and $C$, it 
would be impossible for a convex combination of $(\mu_1,p_1)$ and $(\mu_0,p_0)$ to be in the target region (which is required by \eqref{eq:case0}). 

Observe that: 
\begin{eqnarray}
\norm{(\overline{\mu}(t), \overline{p}(t)) - X} &\leq& \epsilon\sqrt{2} \label{eq:case1-1} \\
\norm{(\mu_1, p_1) - X} &\leq& O(\epsilon) \label{eq:case1-2} \\
\norm{(\mu_0, p_0) - (\mu_1,p_1)} &\geq& \sqrt{2}/16 \label{eq:case1-3}
\end{eqnarray}
where \eqref{eq:case1-1} follows by considering the maximum distance between $X$ and any point in the target region, 
\eqref{eq:case1-2} holds because any vector on the line segment between $B$ and $C$ is $O(\epsilon)$ distance away from $X$, and \eqref{eq:case1-3} holds because the distance between any point on the line segment between $B$ and $C$ and a point in the left (orange) shaded region is at least $\sqrt{2}/16$ (being the distance between the two parallel lines of slope 1). 
Starting from \eqref{eq:case1-1} one has: 
\begin{eqnarray*}
\epsilon\sqrt{2}&\geq& \norm{(\overline{\mu}(t), \overline{p}(t)) - X} \\
&=& \norm{(1/t)(\mu_0, p_0) + (1-1/t)(\mu_1,p_1) - X} \\
&=& \norm{(1/t)[(\mu_0,p_0) - (\mu_1,p_1)] - [X - (\mu_1,p_1)]} \\
&\geq& (1/t)\norm{(\mu_0,p_0) - (\mu_1,p_1)} - \norm{X-(\mu_1,p_1)} \\
&\geq& \sqrt{2}/(16t) - O(\epsilon)
\end{eqnarray*}
where the first equality holds by \eqref{eq:case0}, 
the second-to-last inequality uses the triangle inequality $\norm{W - Z} \geq \norm{W} - \norm{Z}$ for any vectors $W$, $Z$, and the final inequality uses \eqref{eq:case1-2} and \eqref{eq:case1-3}. 
So $\sqrt{2}/(16t) \leq O(\epsilon)$.  It follows that $t \geq \Omega(1/\epsilon)$.

\item Case 2: Suppose $\theta_2(0) > 1/2$. However, suppose $y=z=1/2$.  So $\omega(t) \in \{2, 3\}$, $\pi(2)=\pi(3)=1/2$, and 
$\omega(t)=2$ is the \emph{least efficient} state.   The $h(\mu)$ curve is shown in Fig. \ref{fig:case2}.  Note that $h(1)=1/3$, and so the target point is $X = (1,1/3)$.  The point $A = (1-\epsilon, 1/3+\epsilon)$ is shown in Fig. \ref{fig:case2}.  Point $A$ is one vertex of the small triangular (green) target region 
that defines all points $(\overline{\mu}(t), \overline{p}(t))$ that are $\epsilon$-approximations.  

\begin{figure}[t]
   \centering
   \includegraphics[width=3in]{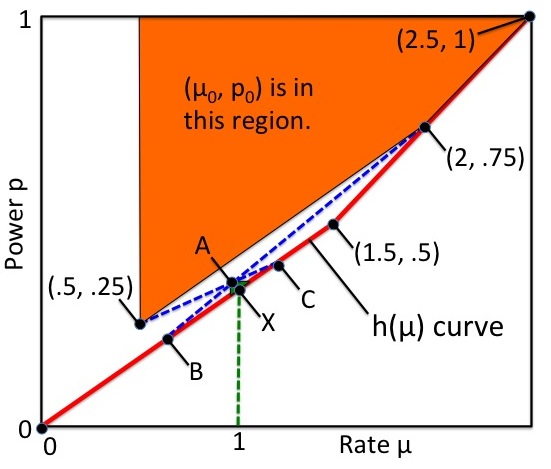} 
   \caption{The performance region for case 2.}
   \label{fig:case2}
\end{figure}

Because $\theta_2(0) \geq 1/2$, the point $(\mu_0, p_0)$ lies somewhere in the (orange) shaded region in Fig. \ref{fig:case2}.  Indeed, if $\theta_2(0)=1/2$, then $(\mu_0, p_0)$ is on the line segment between points $(0.5, 0.25)$ and $(2, 0.75)$. It is above this line segment if $\theta_2(0)>1/2$.   As before, the geometry of the problem ensures an optimal compensation vector $(\mu_1, p_1)$ lies somewhere on the line segment of the $h(\mu)$ curve between points $B$ and $C$ of Fig. \ref{fig:case2}.   As before, it holds that: 
\begin{eqnarray*}
B &=& X - (O(\epsilon), O(\epsilon)) \\
C &=& X + (O(\epsilon), O(\epsilon)) 
\end{eqnarray*}
and: 
\begin{eqnarray*}
\norm{(\overline{\mu}(t), \overline{p}(t)) - X} &\leq& O(\epsilon) \\
\norm{(\mu_1, p_1) - X} &\leq& O(\epsilon) \\
\norm{(\mu_0,p_0)-(\mu_1, p_1)} &\geq& \Theta(1) 
\end{eqnarray*}
As before, it follows that $t \geq \Omega(1/\epsilon)$.

\end{itemize} 

\section{The dynamic algorithm} \label{section:algorithm} 

This section shows that a simple drift-plus-penalty 
algorithm achieves $O(\log(1/\epsilon)/\epsilon)$ convergence time and $O(\log(1/\epsilon))$ average queue size.

\subsection{Problem structure} 

Without loss of generality, assume that $\pi(\omega_k)>0$ for all $k \in \{1, \ldots, M\}$ (else, remove $\omega_k$ from the set $\Omega$).  The value of $\pi(\omega_0)$ is possibly zero. 
For each real number $\mu$ in the interval $[0, \expect{\omega(t)}]$, define $h(\mu)$ as the minimum average 
power required to achieve an average transmission rate of $\mu$.  It is known that $p^*=h(\lambda)$.  Further, it is not difficult to show that 
$h(\mu)$ is non-decreasing, convex, and \emph{piecewise linear} with $h(0)=0$ and $h(\expect{\omega(t)})=1-\pi(\omega_0)$. 
The point $(0,0)$ is a vertex point of the piecewise linear curve $h(\mu)$. 
There are $M$ other vertex points, achieved by the $\omega$-only policies of the form: 
\begin{equation} \label{eq:thresh} p(t) =  \left\{ \begin{array}{ll}
1 &\mbox{ if $\omega(t) \geq \omega_k$} \\
0  & \mbox{ otherwise} 
\end{array}
\right. 
\end{equation} 
for $k \in \{1, \ldots, M\}$. 
This means that a vertex point is achieved by only using 
channel states $\omega(t)$ that are on or above a certain threshold $\omega_k$.    Lowering
the threshold value by selecting a smaller $\omega_k$ allows for a larger $\expect{\mu(t)}$ at the expense
of sometimes using less efficient channel states. 
The proof that this class of policies achieves the vertex points follows by a simple interchange argument that is omitted for brevity. 

For ease of notation, define $\omega_{M+1} \defequiv \infty$ and $\mu_{M+1}\defequiv0$. Let $\{\mu_1, \mu_2, \ldots, \mu_M, \mu_{M+1}\}$ be the set of transmission rates at which there are vertex points.  Specifically, for $k\in\{1, \ldots, M\}$, $\mu_k$ corresponds to the threshold $\omega_k$ in the policy \eqref{eq:thresh}.  That is: 
\begin{equation} \label{eq:mu-k}
 \mu_k \:  \defequiv \: \expect{\omega(t)|\omega(t) \geq \omega_k} = \sum_{i=k}^M \omega_i \pi(\omega_i) 
 \end{equation} 
Note that: 
\[ 0= \mu_{M+1} < \mu_M <  \mu_{M-1} < \cdots < \mu_1 = \expect{\omega(t)}  \]
It follows that $h(\mu_k)$ is the corresponding average power for vertex $k$, so that $(\mu_k, h(\mu_k))$ is a vertex point of the curve $h(\mu)$: 
\begin{equation} \label{eq:h-k}
 h(\mu_k) = Pr[\omega(t) \geq \omega_k] = \sum_{i=k}^M\pi(\omega_i) 
 \end{equation} 

The numbers $\{\mu_1, \mu_2, \ldots, \mu_M, \mu_{M+1}\}$
represent a set of measure 0 in the interval $[0, \expect{\omega(t)}]$.  
It is assumed that the arrival rate $\lambda$ is a number in $[0, \expect{\omega(t)}]$ that 
lies strictly between two points $\mu_{b+1}$ and $\mu_{b}$ for some index $b \in \{1, \ldots, M\}$.   That is: 
\[ \mu_{b+1} < \lambda < \mu_b \]
Thus, the point $(\lambda, h(\lambda))$ can be achieved by timesharing between the vertex points $(\mu_{b+1}, h(\mu_{b+1}))$ and $(\mu_{b}, h(\mu_{b}))$: 
\begin{eqnarray}
\lambda &=& \theta \mu_{b+1} + (1-\theta)\mu_{b} \label{eq:lambda-eq} \\
p^* = h(\lambda) &=& \theta h(\mu_{b+1}) + (1-\theta)h(\mu_{b}) \label{eq:p-eq} 
\end{eqnarray} 
for some probability $\theta$ that satisfies $0<\theta < 1$.  In particular: 
\[ \theta = \frac{\mu_b-\lambda}{\mu_{b} - \mu_{b+1}}    \]

\subsection{The drift-plus-penalty algorithm} 

For each slot $t \in \{0,1,2,\ldots\}$, define $L(t) = \frac{1}{2}Q(t)^2$ and $\Delta(t) = L(t+1) - L(t)$.   Let $V$ be a nonnegative real number.  The \emph{drift-plus-penalty} algorithm from \cite{sno-text}\cite{neely-energy-it} makes a power allocation decision that, every slot $t$, minimizes a bound on $\Delta(t) + Vp(t)$. The value $V$ can be chosen as desired and affects a performance tradeoff.  This technique is known to yield average queue size of $O(V)$ with deviation from optimal average power no more than $O(1/V)$ \cite{sno-text}\cite{neely-energy-it}.  This holds for general multi-queue networks.   By defining $\epsilon = 1/V$, this produces an $O(\epsilon)$ approximation with average queue size $O(1/\epsilon)$.  Further, it can be shown that convergence time is $O(1/\epsilon^2)$ (see Appendix D in \cite{dist-opt-arxiv}).  

In the context of the simple one-queue system of the current paper, the drift-plus-penalty algorithm reduces to the following:  Every slot $t$, observe $Q(t)$ and $\omega(t)$ and choose $p(t) \in \{0,1\}$ to minimize: 
\[ Vp(t) - Q(t)\omega(t)p(t) \]
That is, choose $p(t)$ according to the following rule:  
\begin{equation}\label{eq:dpp}  
p(t) = \left\{ \begin{array}{ll}
1 &\mbox{ if $Q(t)\omega(t) \geq V$} \\
0  & \mbox{ otherwise} 
\end{array}
\right.
\end{equation} 
The current paper shows that, for this special case of a system with only one queue, the above algorithm leads to  
an improved queue size and convergence time tradeoff.  

\subsection{The induced Markov chain}

The drift-plus-penalty algorithm induces a Markov structure on the system.  The system state is 
$Q(t)$ and the state space is the set of nonnegative real numbers.  Observe from \eqref{eq:dpp} that the drift-plus-penalty algorithm has the following behavior: 
\begin{itemize} 
\item $Q(t) \in [V/\omega_{b+1}, V/\omega_b) \implies$ $p(t) = 1$ if and only if $\omega(t) \geq \omega_{b+1}$.  In this case one has (from \eqref{eq:mu-k} and \eqref{eq:h-k}): 
\begin{eqnarray}
\expect{\mu(t) | Q(t) \in [V/\omega_{b+1}, V/\omega_b)} &=& \mu_{b+1} \label{eq:induce1} \\
\expect{p(t) | Q(t) \in [V/\omega_{b+1}, V/\omega_b)} &=& h(\mu_{b+1})\label{eq:induce2}
\end{eqnarray}
\item $Q(t) \in [V/\omega_b, V/\omega_{b-1}) \implies$ $p(t)=1$ if and only if $\omega(t) \geq \omega_b$.  In this case one has: 
\begin{eqnarray}
\expect{\mu(t) | Q(t) \in [V/\omega_{b}, V/\omega_{b-1})} &=& \mu_{b} \label{eq:induce3} \\
\expect{p(t) | Q(t) \in [V/\omega_{b}, V/\omega_{b-1})} &=& h(\mu_{b})\label{eq:induce4}
\end{eqnarray}
\end{itemize} 
where $V/0$ is defined as $\infty$ (in the case $\omega_{b-1}=\omega_0=0$), and 
$\omega_{M+1} = \infty$ so that $V/\omega_{M+1} = 0$. 

Now define intervals $\script{I}^{(1)}, \script{I}^{(2)}, \script{I}^{(3)}, \script{I}^{(4)}$ (see Fig. \ref{fig:drift-line}):
\begin{eqnarray*}
\script{I}^{(1)} &\defequiv& [0, V/\omega_{b+1}) \\
\script{I}^{(2)} &\defequiv& [V/\omega_{b+1}, V/\omega_b)\\
\script{I}^{(3)} &\defequiv& [V/\omega_b, V/\omega_{b-1})  \\
\script{I}^{(4)} &\defequiv& [V/\omega_{b-1}, \infty) 
\end{eqnarray*}
If $V/\omega_{b+1}=0$ then $\script{I}^{(1)}$ is defined as the empty set, and if $V/\omega_{b-1} = \infty$ then $\script{I}^{(4)}$ is defined as the empty set.   The equalities \eqref{eq:induce1}-\eqref{eq:induce4} can be rewritten as: 
\begin{eqnarray}
\expect{\mu(t) | Q(t) \in \script{I}^{(2)}} &=& \mu_{b+1} \label{eq:I-eq-1} \\
\expect{p(t) | Q(t) \in \script{I}^{(2)}} &=& h(\mu_{b+1})\label{eq:I-eq-2} \\
\expect{\mu(t) | Q(t) \in \script{I}^{(3)}} &=& \mu_b \label{eq:I-eq-3} \\
\expect{p(t) | Q(t) \in \script{I}^{(3)}} &=& h(\mu_{b}) \label{eq:I-eq-4} 
\end{eqnarray}

Recall that under the drift-plus-penalty algorithm \eqref{eq:dpp}, if 
$Q(t) \in \script{I}^{(2)}$ then the set of all $\omega(t)$ that lead to a transmission is equal to $\{\omega \in \Omega | \omega \geq \omega_{b+1}\}$. 
 If $Q(t) \in \script{I}^{(1)}$, then the set of all $\omega(t)$ that lead to a transmission depends on the particular value of $Q(t)$.   However, since interval $\script{I}^{(1)}$ is to the left of 
interval $\script{I}^{(2)}$, the set of all $\omega(t)$ that lead to a transmission when $Q(t) \in \script{I}^{(1)}$ is always a subset of $\{\omega \in \Omega | \omega \geq \omega_{b+1}\}$.  Similarly, since $\script{I}^{(4)}$ is to the right of $\script{I}^{(3)}$, the set of all $\omega(t)$ that lead to a transmission when $Q(t) \in \script{I}^{(4)}$ is a superset of the set of all $\omega(t)$ that lead to a  transmission when $Q(t) \in \script{I}^{(3)}$.  Therefore, under the 
drift-plus-penalty algorithm one has: 
\begin{eqnarray}
\expect{\mu(t) | Q(t) \in \script{I}^{(1)}} &\leq& \mu_{b+1} \label{eq:I-ineq-1}  \\
\expect{p(t) | Q(t) \in \script{I}^{(1)}} &\leq& h(\mu_{b+1})\label{eq:I-ineq-2} \\
\expect{\mu(t) | Q(t) \in \script{I}^{(4)}} &\geq& \mu_b \label{eq:I-ineq-3} \\
\expect{p(t) | Q(t) \in \script{I}^{(4)}} &\geq& h(\mu_b) \label{eq:I-ineq-4} 
\end{eqnarray}

For each $i \in \{1, 2, 3, 4\}$ define the indicator function: 
\[ 1\{Q(t) \in \script{I}^{(i)}\} = \left\{ \begin{array}{ll}
1 &\mbox{ if $Q(t) \in \script{I}^{(i)}$} \\
0  & \mbox{ otherwise} 
\end{array}
\right. \]
For each slot $t>0$ and each $i\in\{1,2,3,4\}$, define  $\overline{1}^{(i)}(t)$ as the expected fraction of time that $Q(t) \in \script{I}^{(i)}$: 
\[ \overline{1}^{(i)}(t) \defequiv \frac{1}{t}\sum_{\tau=0}^{t-1} \expect{1\{Q(t) \in \script{I}^{(i)}\}} \]
It follows that (using \eqref{eq:I-eq-2}, \eqref{eq:I-eq-4}, \eqref{eq:I-ineq-2}):
\begin{eqnarray}
\overline{p}(t) &\leq& \overline{1}^{(2)}(t) h(\mu_{b+1}) + \overline{1}^{(3)}(t)h(\mu_b) \nonumber \\
&& + \overline{1}^{(1)}(t)h(\mu_{b+1}) + \overline{1}^{(4)}(t) \label{eq:p-bound} 
\end{eqnarray}
where the final term follows because $p(t) \leq 1$ for all slots $t$.  Similarly (using \eqref{eq:I-eq-1}, \eqref{eq:I-eq-3}, \eqref{eq:I-ineq-1}): 
\begin{eqnarray}
\overline{\mu}(t) &\leq& \overline{1}^{(2)}(t)\mu_{b+1} + \overline{1}^{(3)}(t)\mu_b \nonumber \\
&& + \overline{1}^{(1)}(t)\mu_{b+1}+ \overline{1}^{(4)}(t)\expect{\omega(t)} \label{eq:mu-leq} 
\end{eqnarray}
where the final term follows because $\expect{\mu(t) | Q(t) \in \script{I}^{(4)}} \leq \expect{\omega(t)}$. Likewise (using 
\eqref{eq:I-eq-1}, \eqref{eq:I-eq-3}, \eqref{eq:I-ineq-3}): 
\begin{eqnarray}
\overline{\mu}(t) &\geq& \overline{1}^{(2)}(t)\mu_{b+1} + \overline{1}^{(3)}(t)\mu_b + \overline{1}^{(4)}(t)\mu_b\label{eq:mu-geq} 
\end{eqnarray}
which holds because $\expect{\mu(t) | Q(t) \in \script{I}^{(1)}} \geq 0$. 

\begin{figure}[t]
   \centering
   \includegraphics[width=3.5in]{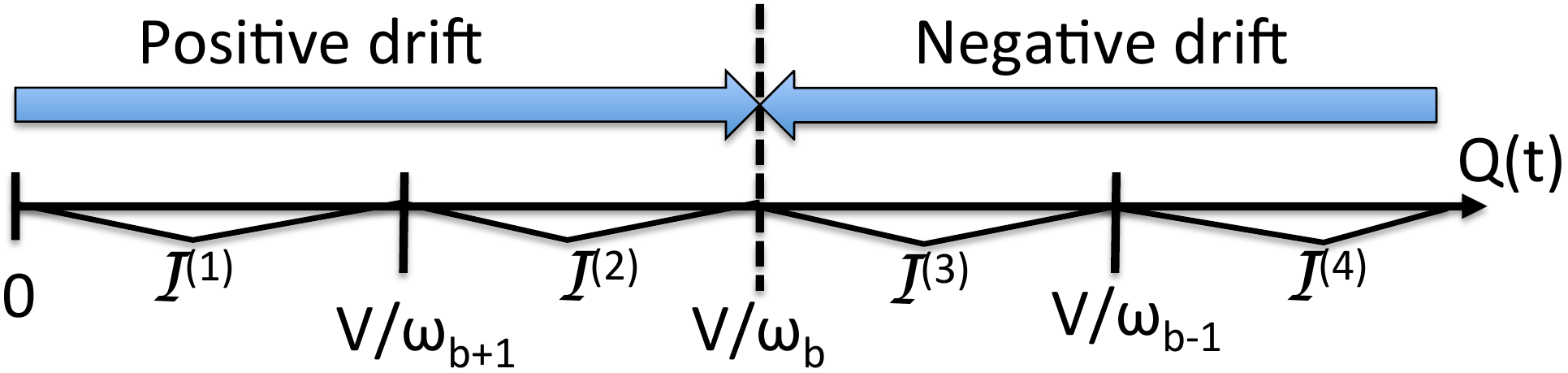} 
   \caption{An illustration of the four intervals $\script{I}_i$ for $i  \in \{1, 2, 3, 4\}$.}
   \label{fig:drift-line}
\end{figure}

In the next section it is shown that: 
\begin{itemize} 
\item $\overline{\mu}(t)$ is close to $\lambda$ when $t$ is sufficiently large. 
\item $\overline{1}^{(1)}(t)$ and $\overline{1}^{(4)}(t)$ are close to $0$ when $t$ and $V$ are sufficiently large. 
\item $\overline{1}^{(2)}(t)$ and $\overline{1}^{(3)}(t)$ are close to $\theta$ and $1-\theta$, respectively, 
when $t$ and $V$ are sufficiently large. 
\item $\overline{p}(t)$ is close to $p^*$ when $t$ and $V$ are sufficiently large. 
\end{itemize} 

Furthermore, to address the issue of convergence time, the notion of ``sufficiently large'' must be made precise.  
A key step 
is establishing bounds on the average queue size. 

\section{Analysis} 

\subsection{The distance between $\overline{\mu}(t)$ and $\lambda$} 

Recall that $\omega_M$ is the largest possible value of $\omega(t)$.  Assume that $V \geq \omega_M^2$. 

\begin{lem} \label{lem:remove-max} If $V \geq \omega_M^2$, then under the drift-plus-penalty algorithm: 

a) One has $p(t)=\mu(t)=0$ whenever 
$Q(t) < \omega_M$.  

b) The queueing equation \eqref{eq:q-update} can be replaced by the following for all slots $t \in \{0,1,2, \ldots\}$: 
\[ Q(t+1) = Q(t) + a(t) - \mu(t) \]
\end{lem} 
\begin{proof} 
Suppose $V \geq \omega_M^2$.  To prove (a), suppose that $Q(t) < \omega_M$. Since $\omega(t) \leq \omega_M$ for all $t$, one has: 
\begin{eqnarray*}
Q(t)\omega(t) &\leq& Q(t) \omega_M \\
&<& \omega_M^2 \\
&\leq& V
\end{eqnarray*}
and so the algorithm \eqref{eq:dpp} chooses $p(t)=0$, so that $\mu(t)$ is also $0$. This proves part (a).

To prove (b), note that part (a) implies $Q(t) \geq \mu(t)$ for all slots $t$. 
Indeed, this holds in the case $Q(t) < \omega_M$ (since part (a) ensures $\mu(t)=0$ in this case), 
and also holds in the case $Q(t) \geq \omega_M$ (since $\omega_M \geq \mu(t)$ always). 
Thus: 
\begin{eqnarray*}
Q(t+1) &=& \max[Q(t) + a(t) - \mu(t),0] \\
&=& Q(t) + a(t) - \mu(t) 
\end{eqnarray*}
\end{proof}

\begin{lem} \label{lem:mu-equality} If $V \geq \omega_M^2$ and $Q(0)=q_0$ with probability 1 (for some constant 
$q_0\geq 0$), then for every slot $t>0$: 
\[  \overline{\mu}(t) = \lambda - \expect{Q(t)- q_0}/t \]
\end{lem} 
\begin{proof} 
By Lemma \ref{lem:remove-max} one has for all slots $\tau  \in \{0, 1, 2, \ldots\}$: 
\[ Q(\tau+1) - Q(\tau) = a(\tau) - \mu(\tau) \]
Taking expectations gives: 
\[ \expect{Q(\tau+1)} - \expect{Q(\tau)} = \lambda - \expect{\mu(\tau)} \]
Summing the above over $\tau \in \{0, 1, 2, \ldots, t-1\}$ gives: 
\[ \expect{Q(t)} - \expect{Q(0)} = \lambda t - \sum_{\tau=0}^{t-1} \expect{\mu(\tau)} \]
Dividing by $t$ proves the result. 
\end{proof} 

The above lemma implies that if $V \geq \omega_M^2$, then 
$\overline{\mu}(t)$ converges to $\lambda$ whenever $\expect{Q(t)-q_0}/t$ converges to 0. 

\subsection{The distance between $\overline{1}^{(2)}(t)$ and $\theta$} 

The following lemma shows that if $\overline{1}^{(1)}(t)$, $\overline{1}^{(4)}(t)$, and $\expect{Q(t)-q_0}/t$ are close to $0$, then $\overline{1}^{(2)}(t)$ is close to $\theta$. 

\begin{lem} \label{lem:1bound} If $V\geq \omega_M^2$ and $Q(0)=q_0$ with probability 1 (for some constant $q_0 \geq 0$), 
then for all slots $t>0$: 
\begin{eqnarray*}
 \theta - \frac{[\mu_b \overline{1}^{(1)}(t)-\psi(t)]}{\mu_b - \mu_{b+1}} &\leq& \overline{1}^{(2)}(t)\\
 &\leq& \theta + \frac{\overline{1}^{(4)}(t)\expect{\omega(t)} + \psi(t)}{\mu_b-\mu_{b+1}}
 \end{eqnarray*}
 where $\psi(t)$ is defined: 
 \[ \psi(t) \defequiv \expect{Q(t)-q_0}/t \]
\end{lem} 

\begin{proof} 
Fix $t>0$.  Lemma \ref{lem:mu-equality} implies: 
\begin{eqnarray*}
\lambda &=& \overline{\mu}(t)  + \psi(t)\\
&\geq& \overline{1}^{(2)}(t)\mu_{b+1} + \overline{1}^{(3)}(t) \mu_b + \overline{1}^{(4)}(t)\mu_b + \psi(t) \\
&=&  \overline{1}^{(2)}(t)\mu_{b+1} + (1-\overline{1}^{(2)}(t))\mu_b - \overline{1}^{(1)}(t)\mu_b + \psi(t)
\end{eqnarray*}
where the first inequality holds by \eqref{eq:mu-geq}.  
Substituting the identity for $\lambda$ given in \eqref{eq:lambda-eq} into the above inequality gives: 
\begin{align*}
&\theta\mu_{b+1} + (1-\theta)\mu_b \\
&\geq  \overline{1}^{(2)}(t)\mu_{b+1} + (1-\overline{1}^{(2)}(t))\mu_b - \overline{1}^{(1)}(t)\mu_b + \psi(t) 
\end{align*}
Rearranging terms proves that: 
\[ \theta - \frac{[\mu_b \overline{1}^{(1)}(t)-\psi(t)]}{\mu_b - \mu_{b+1}} \leq \overline{1}^{(2)}(t)  \]

To prove the second inequality, note that: 
\begin{eqnarray}
\lambda &=& \overline{\mu}(t) + \psi(t) \label{eq:dude1} \\
&\leq& \overline{1}^{(2)}(t) \mu_{b+1} + \overline{1}^{(3)}(t)\mu_b + \overline{1}^{(1)}(t)\mu_{b+1} \nonumber \\
&& + \overline{1}^{(4)}(t)\expect{\omega(t)} + \psi(t)  \label{eq:dude2} \\
&=& \overline{1}^{(2)}(t) \mu_{b+1} + (1-\overline{1}^{(2)}(t))\mu_b \nonumber \\
&& - \overline{1}^{(1)}(t)\mu_b - \overline{1}^{(4)}(t)\mu_b + \overline{1}^{(1)}(t)\mu_{b+1}  \nonumber \\
&&  + \overline{1}^{(4)}(t)\expect{\omega(t)} +  \psi(t)  \nonumber \\
&\leq& \overline{1}^{(2)}(t) \mu_{b+1} + (1-\overline{1}^{(2)}(t))\mu_b \nonumber \\
&&+ \overline{1}^{(4)}(t)\expect{\omega(t)} + \psi(t) \label{eq:dude3}
\end{eqnarray} 
where \eqref{eq:dude1} holds by Lemma \ref{lem:mu-equality}, \eqref{eq:dude2} holds by \eqref{eq:mu-leq}, 
and  \eqref{eq:dude3} holds because $\mu_{b+1} < \mu_b$.   Substituting the identity for 
$\lambda$ given in \eqref{eq:lambda-eq} gives: 
\begin{eqnarray*}
\theta\mu_{b+1} + (1-\theta)\mu_b &\leq& \overline{1}^{(2)}(t)\mu_{b+1} + (1-\overline{1}^{(2)}(t))\mu_b \\
&& + \overline{1}^{(4)}(t)\expect{\omega(t)} +\psi(t)
\end{eqnarray*}
Rearranging terms proves the result. 
\end{proof}

\subsection{Positive and negative drift} 

Define $\expect{Q(t+1) - Q(t)|Q(t)}$ as the \emph{conditional drift}.  Assume that $V\geq \omega_M^2$, 
so that Lemma \ref{lem:remove-max} implies $Q(t+1) - Q(t) = a(t) - \mu(t)$ for all slots $t$.  Thus: 
\begin{eqnarray*}
\expect{Q(t+1)-Q(t)|Q(t)} &=& \expect{a(t) - \mu(t) | Q(t)} \\
&=& \lambda - \expect{\mu(t)|Q(t)} 
\end{eqnarray*}
where the final equality follows because $a(t)$ is independent of $Q(t)$. 
From \eqref{eq:I-eq-1} and \eqref{eq:I-ineq-1} one has for all slots $t$: 
\begin{eqnarray*}
\expect{\mu(t) | Q(t) < V/\omega_b} \leq \mu_{b+1}
\end{eqnarray*}
Likewise, from \eqref{eq:I-eq-3} and \eqref{eq:I-ineq-3} one has: 
\begin{eqnarray*}
\expect{\mu(t) | Q(t) \geq V/\omega_b} \geq \mu_{b}  
\end{eqnarray*}
Define positive constants $\beta_L$ and $\beta_R$ (associated with drift when $Q(t)$ is to the \emph{Left} and \emph{Right} of the threshold $V/\omega_b$) by: 
\begin{eqnarray*}
\beta_L \defequiv \lambda - \mu_{b+1} \: \: \: \: , \: \: \: \: 
\beta_R \defequiv \mu_b - \lambda
\end{eqnarray*}
It follows that: 
\begin{eqnarray}
\expect{Q(t+1) - Q(t) | Q(t)} \geq \beta_L &   \mbox{ if $Q(t) < V/\omega_b$} \label{eq:pos-drift-Q} \\
\expect{Q(t+1) - Q(t) | Q(t)} \leq -\beta_R &  \mbox{ if $Q(t) \geq V/\omega_b$} \label{eq:neg-drift-Q}
\end{eqnarray} 
In particular, the system has \emph{positive drift} if $Q(t)< V/\omega_b$, and \emph{negative drift} otherwise (see Fig. \ref{fig:drift-line}). 

\subsection{A basic drift lemma} 

Consider a real-valued random process $Z(t)$ over slots $t \in \{0, 1, 2, \ldots\}$.  
The following drift lemma is similar in spirit 
to results in \cite{martingale-concentration}\cite{longbo-lagrange-tac}, but focuses on a finite time horizon with an arbitrary  initial condition $Z(0)=z_0$ (rather than on steady state), and on expectations at a given time (rather than time averages).    These distinctions are crucial to convergence time analysis. The lemma will be applied using $Z(t) = Q(t)$ for bounds on average queue size and on $\overline{1}^{(4)}(t)$.  It will then be applied using 
$Z(t) = V/\omega_b-Q(t)$ to bound $\overline{1}^{(1)}(t)$.  Assume there is a constant $\delta_{max}>0$ such that with probability 1:

\begin{equation} \label{eq:z-delta} 
|Z(t+1) - Z(t)| \leq \delta_{max} \: \: \forall t \in \{0, 1, 2, \ldots\} 
\end{equation} 
Suppose  there are constants $\theta \in \mathbb{R}$ and $\beta>0$ such that: 
\begin{equation} \label{eq:drift-lem-cond} 
\expect{Z(t+1) - Z(t) | Z(t)} \leq  \left\{ \begin{array}{ll}
\delta_{max} &\mbox{ if $Z(t) < \theta$} \\
-\beta  & \mbox{ if $Z(t) \geq \theta$} 
\end{array}
\right.
\end{equation} 
Note that if \eqref{eq:z-delta} holds then \eqref{eq:drift-lem-cond} automatically holds for the special case $Z(t)<\theta$. Thus, the negative drift case $Z(t) \geq \theta$ is the important case for condition \eqref{eq:drift-lem-cond}.  Further, if 
 \eqref{eq:z-delta}-\eqref{eq:drift-lem-cond} both hold, then the constant $\beta$ necessarily satisfies: 
\[ 0 < \beta \leq \delta_{max} \] 

\begin{lem} \label{lem:drift} 
Suppose $Z(t)$ is a random process that 
satisfies \eqref{eq:z-delta}-\eqref{eq:drift-lem-cond} for given constants $\theta$, $\delta_{max}$, $\beta$ (with $\theta\in \mathbb{R}$ and 
$0 < \beta \leq \delta_{max}$).  Suppose $Z(0)=z_0$ for some $z_0\in\mathbb{R}$. 
Then for every slot $t\geq0$ the following holds: 
\begin{eqnarray} 
\expect{e^{rZ(t)}} \leq D + \left(e^{rz_0}-D\right)\rho^t   \label{eq:nice-ineq} 
  \end{eqnarray} 
where constants $r$, $\rho$, $D$ are defined: 
\begin{eqnarray} 
r &\defequiv& \frac{\beta}{\delta_{max}^2 + \delta_{max}\beta/3} \label{eq:r}  \\
\rho &\defequiv& 1 - r\beta/2 \label{eq:rho} \\
D &\defequiv& \frac{(e^{r\delta_{max}} - \rho) e^{r\theta}}{1-\rho} \label{eq:D}
\end{eqnarray} 
\end{lem} 

Note that the property $0 < \beta \leq \delta_{max}$ can be used to show that $0 < \rho < 1$. 

\begin{proof} (Lemma \ref{lem:drift}) 
The proof is by induction. The inequality \eqref{eq:nice-ineq} trivially holds for $t=0$. Suppose \eqref{eq:nice-ineq} 
holds at some slot $t \geq 0$.  The goal is to show that it also holds on slot $t+1$. 
Let $r$ be a positive 
number that satisfies $0 <  r\delta_{max} < 3$.  It is known from results in \cite{martingale-concentration}
that for any real number $x$ that satisfies $|x|\leq \delta_{max}$: 
\begin{equation} \label{eq:e-ineq}
 e^{rx}  \leq 1 + rx + \frac{(r\delta_{max})^2}{2(1-r\delta_{max}/3)} 
 \end{equation} 
 Define $\delta(t) = Z(t+1) - Z(t)$ and note that $|\delta(t)|\leq \delta_{max}$ for all $t$. 
Then: 
\begin{eqnarray}
e^{rZ(t+1)} &=& e^{rZ(t)}e^{r\delta(t)}  \nonumber \\
&\leq& e^{rZ(t)}\left[1 + r\delta(t) +  \frac{(r\delta_{max})^2}{2(1-r\delta_{max}/3)}\right]  \label{eq:baz1} 
\end{eqnarray}
where the final inequality holds by \eqref{eq:e-ineq}. Choose $r$ such that: 
\begin{equation} \label{eq:choose-r} 
\frac{(r\delta_{max})^2}{2(1-r\delta_{max}/3)} \leq \frac{r\beta}{2} 
\end{equation} 
It is not difficult to show that the value of $r$ given in \eqref{eq:r} simultaneously satisfies \eqref{eq:choose-r} and  $0<r\delta_{max} < 3$. 
For this value of $r$, substituting \eqref{eq:choose-r} into \eqref{eq:baz1} gives: 
\begin{eqnarray}
e^{rZ(t+1)}  &\leq&  e^{rZ(t)}\left[1 + r\delta(t) +\frac{r\beta}{2}\right] \label{eq:yup} 
\end{eqnarray}
Now consider the following two cases: 
\begin{itemize} 
\item Case 1:  Suppose $Z(t) \geq \theta$.  Taking conditional expectations of 
\eqref{eq:yup} gives: 
\begin{eqnarray}
\expect{e^{rZ(t+1)}|Z(t)} &\leq& \expect{e^{rZ(t)}[1 + r\delta(t) + \frac{r\beta}{2}]|Z(t)} \nonumber \\
&\leq& e^{rZ(t)}[1- r\beta + \frac{r\beta}{2}] \label{eq:yup1} \\
&=& e^{rZ(t)}\rho \nonumber
\end{eqnarray}
where \eqref{eq:yup1} follows by \eqref{eq:drift-lem-cond}, and the final equality holds by 
definition of $\rho$ in 
\eqref{eq:rho}. 
\item Case 2:  Suppose $Z(t) < \theta$.  Then: 
\begin{eqnarray*}
\expect{e^{rZ(t+1)}| Z(t)} &=& \expect{e^{rZ(t)}e^{r\delta(t)}} \\
&\leq&e^{rZ(t)} e^{r\delta_{max}} 
\end{eqnarray*}
\end{itemize} 
Putting these two cases together gives: 
\begin{eqnarray*}
&&\hspace{-.4in}\expect{e^{rZ(t+1)}} \\
&\leq& \rho\expect{e^{rZ(t)}| Z(t) \geq \theta}Pr[Z(t)\geq \theta] \\
&&+ e^{r\delta_{max}}\expect{e^{rZ(t)}|Z(t)<\theta}Pr[Z(t)<\theta]\\
&=&\rho \expect{e^{rZ(t)}} \\
&& + (e^{r\delta_{max}}-\rho) \expect{e^{rZ(t)}|Z(t)<\theta}Pr[Z(t)<\theta] \\
&\leq& \rho\expect{e^{rZ(t)}} + (e^{r\delta_{max}} - \rho)e^{r\theta}
\end{eqnarray*}
where the final inequality uses the fact that $e^{r\delta_{max}}>  1 > \rho$.   By the induction assumption it is known that  \eqref{eq:nice-ineq} holds on slot $t$.  Substituting 
 \eqref{eq:nice-ineq} into the 
 right-hand-side of the above inequality gives: 
 \begin{eqnarray*}
\expect{e^{rZ(t+1)}} &\leq& \rho\left[D + \left(e^{rz_0}-D\right)\rho^t\right] \\
&& +  (e^{r\delta_{max}} - \rho)e^{r\theta} \\
 &=& D + \left(e^{rz_0} - D\right)\rho^{t+1} 
 \end{eqnarray*}
 where the final equality holds by the definition of $D$ in \eqref{eq:D}. This completes the induction step.
\end{proof} 

Let $1\{Z(\tau) \geq \theta + c\}$ be an indicator function that is 1 if $Z(\tau) \geq \theta + c$, and $0$ else. The next corollary shows that the 
expected fraction of time that this indicator is 1 decays exponentially in $c$.

\begin{cor} \label{corollary} If the  assumptions of Lemma \ref{lem:drift} hold, then  for any $c>0$ and any slots $T$ and $t$ that
satisfy $0 \leq T <t$: 
\begin{align}
&\frac{1}{t}\sum_{\tau=0}^{t-1} \expect{1\{Z(\tau)\geq \theta + c\}} \nonumber \\
&\leq \frac{(e^{r\delta_{max}}-\rho)e^{-rc}}{1-\rho} +\left[\frac{T}{t} + \frac{e^{r(z_0-c-\theta)}\rho^T}{t(1-\rho)}\right]\label{eq:cor2a} 
\end{align}
where $r$ and $\rho$ are defined in \eqref{eq:r}-\eqref{eq:rho}.  Further, if $z_0 \leq \theta$ then for any $t>0$: 
\begin{align}
\frac{1}{t}\sum_{\tau=0}^{t-1} \expect{1\{Z(\tau)\geq \theta + c\}} \leq \frac{e^{-rc} (e^{r\delta_{max}}-\rho+1/t)}{(1-\rho)} \label{eq:cor2b}
\end{align}
\end{cor} 

The intuition behind the right-hand-side of \eqref{eq:cor2a} is that the first term  represents a ``steady state'' bound as $t\rightarrow\infty$, which decays like $e^{-rc}$.  The last two terms (in brackets) are due to the transient effect of the initial condition $z_0$.  This transient can be significant when $z_0 > \theta$.  In that case, $e^{r(z_0-c-\theta)}$ might be large, and a time $T$ is required to shrink this term by multiplication with the  factor $\rho^T$. 

\begin{proof} (Corollary \ref{corollary}) 
One has for $t >T$: 
\begin{align}
\sum_{\tau=0}^{t-1} \expect{1\{Z(\tau) \geq \theta + c\}}  \leq T + \sum_{\tau=T}^{t-1} \expect{1\{Z(\tau) \geq \theta + c\}} \label{eq:cor1-here} 
\end{align}
However, for every slot $\tau\geq0$ one has:  
\[  e^{rZ(\tau)} \geq e^{r(\theta+c)}1\{Z(\tau)\geq \theta + c\}   \]
Taking expectations of both sides gives: 
\[ \expect{e^{rZ(\tau)}} \geq e^{r(\theta + c)} \expect{1\{Z(\tau) \geq \theta + c\}} \]
Rearranging the above shows that for every slot $\tau\geq0$: 
\begin{eqnarray*}
\expect{1\{Z(\tau) \geq \theta + c\}} &\leq& e^{-r(\theta+c)}\expect{e^{rZ(\tau)}} \\
&\leq& e^{-r(\theta+c)}[D + (e^{rz_0} - D)\rho^\tau)]
\end{eqnarray*}
where the final inequality uses \eqref{eq:nice-ineq}.  Substituting the above inequality into the right-hand-side of 
\eqref{eq:cor1-here} gives: 
\begin{align}
&\sum_{\tau=0}^{t-1} \expect{1\{Z(\tau) \geq \theta + c\}} \nonumber \\
&\leq T + e^{-r(\theta+c)}\sum_{\tau=T}^{t-1}\left[D + (e^{rz_0}-D)\rho^{\tau}  \right] \nonumber \\
&= T + e^{-r(\theta+c)}\left[(t-T)D + (e^{rz_0}-D)\rho^T\frac{(1-\rho^{t-T})}{(1-\rho)}\right] \nonumber \\
&\leq T + e^{-r(\theta+c)}\left[tD + \frac{e^{rz_0}\rho^T}{(1-\rho)}\right] \nonumber 
\end{align}
Dividing by $t$ and substituting the definition of $D$ proves  \eqref{eq:cor2a}.  Inequality \eqref{eq:cor2b} follows immediately from \eqref{eq:cor2a} by choosing $T=0$. 
\end{proof} 

\subsection{Bounding $\expect{Q(t)}$ and $\overline{1}^{(4)}(t)$} 

Let $Q(t)$ be the backlog process under the drift-plus-penalty algorithm. Assume that $V \geq \omega_M^2$ and the initial condition is $Q(0)=q_0$ for 
some constant $q_0$. Define $\delta_{max} \defequiv \max[\omega_M, a_{max}]$ as the largest possible change in $Q(t)$ over one slot, so that: 
\[ |Q(t+1)-Q(t)| \leq \delta_{max} \: \: \forall t \in \{0, 1,2, \ldots\} \]
From \eqref{eq:neg-drift-Q} it holds that: 
\[ \expect{Q(t+1) - Q(t)|Q(t)} \leq  \left\{ \begin{array}{ll}
\delta_{max} &\mbox{ if $Q(t) < V/\omega_b$} \\
-\beta_R  & \mbox{ if $Q(t) \geq V/\omega_b$} 
\end{array}
\right.
\]
It follows that the process $Q(t)$ satisfies the conditions \eqref{eq:z-delta}-\eqref{eq:drift-lem-cond} required for Lemma \ref{lem:drift}. 
Specifically, define $Z(t) = Q(t)$, $z_0=q_0$, $\theta = V/\omega_b$, $\beta = \beta_R$.

\begin{lem} \label{lem:q-bound}  If $0 \leq q_0 \leq V/\omega_b$ and  $V \geq \omega_M^2$, then for all slots $t\geq0$
one has: 
\begin{equation*}  
 \expect{Q(t)} \leq \frac{V}{\omega_b} + \frac{1}{r_R} \log\left(1 + \frac{e^{r_R\delta_{max}}-\rho_R}{1-\rho_R}\right) = O(V)
\end{equation*} 
where constants $r_R$ and $\rho_R$ are defined: 
\begin{eqnarray} 
r_R &\defequiv& \frac{\beta_R}{\delta_{max}^2 + \delta_{max}\beta_R/3}  \label{eq:r4} \\
\rho_R &\defequiv& 1 - r_R\beta_R/2 \label{eq:rho4} 
\end{eqnarray} 
\end{lem} 

The lemma provides a bound on $\expect{Q(t)}$ that does not depend on $t$.  
The bound holds whenever the initial condition satisfies 
$0 \leq q_0 \leq V/\omega_b$.  Typically, the initial condition is $q_0=0$.  However, a \emph{place-holder} 
technique in Section \ref{section:place-holders} requires a nonzero initial condition that still satisfies the desired
inequality $0 \leq q_0 \leq V/\omega_b$. 

\begin{proof} For ease of notation, let ``$r$'' and ``$\rho$'' respectively 
denote ``$r_R$'' and ``$\rho_R$'' given in \eqref{eq:r4} and \eqref{eq:rho4}. 
Define $\theta=V/\omega_b$ and 
$\beta=\beta_R$.  
By \eqref{eq:nice-ineq} one has for all $t\geq0$ (using $Z(0)=Q(0)=q_0$): 
\begin{eqnarray*}
\expect{e^{rQ(t)}} &\leq& D + (e^{rq_0}-D)\rho^t \\
&\leq& D + e^{rV/\omega_b}
\end{eqnarray*}
where $D$ is given in \eqref{eq:D}, and 
where the final inequality uses  $D\rho^t \geq 0$ and $q_0 \leq V/\omega_b$. 
Using Jensen's inequality gives: 
\[ e^{r\expect{Q(t)}} \leq D +e^{rV/\omega_b} \]
Taking a log of both sides and dividing by $r$ gives: 
\begin{eqnarray*}
  \expect{Q(t)} &\leq& \frac{\log(D+e^{rV/\omega_b})}{r} \\
  &=& \frac{1}{r}\log\left(e^{rV/\omega_b} + \frac{(e^{r\delta_{max}} - \rho)e^{rV/\omega_b}}{1-\rho}\right) \\
  &=& \frac{V}{\omega_b} + \frac{1}{r}\log\left(1 + \frac{e^{r\delta_{max}}-\rho}{1-\rho} \right)
    \end{eqnarray*}
\end{proof}

\begin{lem} \label{lem:1boundzzz} If $0 \leq q_0 \leq V/\omega_b$ and $V \geq \omega_M^2$,  then for all slots $t>0$: 
\[ \overline{1}^{(4)}(t) \leq O(e^{-r_RV(\frac{1}{\omega_{b-1}}- \frac{1}{\omega_b})})\]
where $r_R$ is given by \eqref{eq:r4}. 
\end{lem}

\begin{proof} 
For ease of notation, this proof uses ``$r$'' to denote ``$r_R$.'' 
If the interval $\script{I}^{(4)}$ does not exist then $\overline{1}^{(4)}(t)=0$ and the result is trivial.  Now suppose
interval $\script{I}^{(4)}$ exists (so that the interval $\script{I}^{(3)}$ is not the final interval in Fig. \ref{fig:drift-line}). 
Define $\theta = V/\omega_b$, $c = V(1/\omega_{b-1} - 1/\omega_{b})$, $\beta = \beta_R$, $\rho = 1-r\beta_R/2$. Then $1\{Q(\tau) \geq \theta + c\} = 1$ if and only if $Q(\tau) \geq V/\omega_{b-1}$, which holds if and only if $Q(\tau) \in \script{I}_4$.  Thus, for all slots $t>0$: 
\begin{eqnarray}
\overline{1}^{(4)}(t) &=& \frac{1}{t}\sum_{\tau=0}^{t-1} \expect{1\{Q(\tau)\geq \theta + c\}} \nonumber \\
&\leq& \frac{ e^{-rc} (e^{r\delta_{max}} -\rho + 1/t)}{1-\rho} \label{eq:bat1} \\
&=& \frac{e^{-rV(\frac{1}{\omega_{b-1}}-\frac{1}{\omega_b})} (e^{r\delta_{max}}-\rho+1/t)}{1-\rho} \nonumber
\end{eqnarray}
where \eqref{eq:bat1} holds by \eqref{eq:cor2b} (which applies since $z_0=q_0 \leq \theta$).  The right-hand-side of the above inequality is indeed of the form $O(e^{-rV(\frac{1}{\omega_{b-1}}- \frac{1}{\omega_b})})$. 
\end{proof}

\subsection{Bounding $\overline{1}^{(1)}(t)$} 

One can similarly prove a bound on $\overline{1}^{(1)}(t)$.  The intuition is that the positive drift in region $\script{I}^{(2)}$ of Fig. \ref{fig:drift-line}, together with the fact that the size of interval $\script{I}^{(2)}$ is $\Theta(V)$, makes the fraction of time the queue is to the left of $V/\omega_b$ decay exponentially as we move further left. 
The result is given below.  Recall that $Q(0)=q_0$ for some constant $q_0\geq 0$. 

\begin{lem} \label{lem:1boundneg} If $q_0\geq 0$ and $V \geq \omega_M^2$, then for all slots $t>0$ one has: 
\[ \overline{1}^{(1)}(t) \leq   O(V)/t + O(e^{-r_LV(\frac{1}{\omega_{b}} - \frac{1}{\omega_{b+1}})}) \]
where $r_L$ is defined: 
\begin{equation*}
r_L \defequiv \frac{\beta_L}{\delta_{max}^2 + \delta_{max} \beta_L/3}
\end{equation*} 
\end{lem} 

Intuitively, the first term in the above lemma (that is, the $O(V)/t$ term) 
bounds the contribution from the \emph{transient time} 
starting from the initial state $Q(0)=q_0$ and ending when the threshold $V/\omega_b$ is crossed.  
The second term 
represents a  ``steady state'' probability assuming an initial condition $V/\omega_b$.  The proof defines a new
process $Z(t) = V/\omega_b-Q(t)$.  It then applies inequality \eqref{eq:cor2a}  
of Corollary \ref{corollary}, with a suitably large time $T>0$,  to handle the initial condition $z_0=V/\omega_b-q_0$. 

\begin{proof} (Lemma \ref{lem:1boundneg})  Define $Z(t) = V/\omega_b - Q(t)$ and note that  $|Z(t+1)-Z(t)|\leq \delta_{max}$ still holds. Further, from \eqref{eq:pos-drift-Q} it holds: 
\begin{eqnarray*}
\expect{Z(t+1) - Z(t)|Z(t)} \leq  \left\{ \begin{array}{ll}
\delta_{max} &\mbox{ if $Z(t) \leq 0$} \\
-\beta_L  & \mbox{ if $Z(t) >0$} 
\end{array}
\right.
\end{eqnarray*}
Now define $\theta$ as any positive value.  It follows that: 
\begin{eqnarray*}
\expect{Z(t+1) - Z(t)|Z(t)} \leq  \left\{ \begin{array}{ll}
\delta_{max} &\mbox{ if $Z(t) < \theta$} \\
-\beta_L  & \mbox{ if $Z(t) \geq \theta$} 
\end{array}
\right.
\end{eqnarray*}
Thus, the conditions \eqref{eq:z-delta}-\eqref{eq:drift-lem-cond} hold for this $Z(t)$ process, with initial condition 
$z_0 = V/\omega_b - q_0$.  Therefore, Corollary \ref{corollary} can be applied. 

For ease of notation let ``$r$'' represent ``$r_L$,'' let ``$\beta$'' represent ``$\beta_R$,'' and let ``$\rho$'' 
represent ``$\rho_L$,'' where $\rho_L \defequiv 1- r_L\beta_L/2$.  Define $c = V/\omega_b - V/\omega_{b+1}$. 
From \eqref{eq:cor2a} of Corollary \ref{corollary}, the following holds
for all slots $T$, $t$ such that $0 \leq T < t$: 
\begin{align*}
&\frac{1}{t}\sum_{\tau=0}^{t-1} \expect{1\{Z(\tau)\geq \theta+c\}} \\
&\leq \frac{(e^{r\delta_{max}}-\rho)e^{-rc}}{1-\rho} + \left[\frac{T}{t} + \frac{e^{r(z_0-c-\theta)}\rho^T}{t(1-\rho)}\right]
\end{align*}
This holds for all $\theta>0$. 
Taking a limit as $\theta\rightarrow 0^+$ gives: 
\begin{align*}
&\frac{1}{t}\sum_{\tau=0}^{t-1} \expect{1\{Z(\tau)> c\}} \\
&\leq \frac{(e^{r\delta_{max}}-\rho)e^{-rc}}{1-\rho} + \left[\frac{T}{t} + \frac{e^{r(z_0-c)}\rho^T}{t(1-\rho)}\right]
\end{align*}
Notice that the event $1\{Z(\tau)>c\}$ is equivalent to the event $\{Q(\tau)<V/\omega_{b+1}\}$, which is the same
as the event $Q(\tau) \in \script{I}_1$ (see Fig. \ref{fig:drift-line}). Thus, the left-hand-side of the above inequality is the same as $\overline{1}^{(1)}(t)$.  Hence: 
\begin{eqnarray*}
\overline{1}^{(1)}(t) &\leq& \frac{(e^{r\delta_{max}}-\rho)e^{-rc}}{1-\rho} + \left[\frac{T}{t} + \frac{e^{r(z_0-c)}\rho^T}{t(1-\rho)}\right] \\
&\leq& \frac{(e^{r\delta_{max}}-\rho)e^{-rc}}{1-\rho} + \left[\frac{T}{t} + \frac{e^{rV/\omega_{b+1}}\rho^T}{t(1-\rho)}\right] 
\end{eqnarray*}
where the final inequality uses the fact that $z_0 \leq V/\omega_b$.  By definition of $c$, the first term on the right-hand-side is 
$O(e^{-r_LV(\frac{1}{\omega_b}-\frac{1}{\omega_{b+1}})})$. It remains to choose a value $T>0$ for which the remaining two terms (in brackets) are $O(V)/t$.   
To this end, define $x\defequiv r/(\omega_{b+1}\log(1/\rho_L))$. 
Choose $T$ as the smallest integer that is greater than or equal to $x V$.  Then $T = O(V)$ and: 
\begin{eqnarray*}
\frac{T}{t} &\leq& O(V)/t \\
\frac{e^{rV/\omega_{b+1}}\rho^T}{t(1-\rho)} &\leq& \frac{e^{rV/\omega_{b+1}}\rho^{x V}}{t(1-\rho)}\\
&=& \frac{1}{t(1-\rho)}\\
&\leq& O(V)/t
\end{eqnarray*}
\end{proof}

\subsection{Optimal backlog and near-optimal convergence time} 

Define: 
\[ \gamma \: \defequiv \: \min\left[r_R\left(\frac{1}{\omega_{b-1}}-\frac{1}{\omega_b}\right), r_L\left(\frac{1}{\omega_b}-\frac{1}{\omega_{b+1}}\right)\right] \]
Results of Lemmas \ref{lem:q-bound}-\ref{lem:1boundneg} imply that if the drift-plus-penalty 
algorithm \eqref{eq:dpp} is used with $V \geq \omega_M^2$, and if the initial queue state satisfies
$0 \leq q_0\leq V/\omega_b$, then for all $t>0$: 
\begin{eqnarray}
\overline{Q}(t) &\leq& O(V) \label{eq:yep1}  \\
\expect{Q(t)}/t &\leq& O(V)/t \label{eq:yep2}\\
\overline{1}^{(4)}(t) &\leq& O(e^{-\gamma V})\label{eq:yep3}  \\
\overline{1}^{(1)}(t) &\leq& O(e^{-\gamma V}) + O(V)/t \label{eq:yep4} 
\end{eqnarray}
Indeed, \eqref{eq:yep1}-\eqref{eq:yep2} follow from Lemma \ref{lem:q-bound}, while \eqref{eq:yep3} and 
\eqref{eq:yep4} follow from Lemmas \ref{lem:1boundzzz} and \ref{lem:1boundneg}, respectively.

Fix $\epsilon>0$ and define: 
\begin{eqnarray*}
 V &=& \max[(1/\gamma)\log(1/\epsilon), \omega_M^2] \\
 T_{\epsilon} &=& \log(1/\epsilon)/\epsilon 
 \end{eqnarray*}
Inequalities \eqref{eq:yep1}-\eqref{eq:yep4} can be used to easily derive the following facts: 
\begin{itemize} 
\item Fact 1: For all slots $t>0$ one has $\overline{Q}(t) \leq O(\log(1/\epsilon))$.
\item Fact 2: For all slots $t>T_{\epsilon}$ one has $\expect{Q(t)}/t \leq O(\epsilon)$. 
\item Fact 3: For all slots $t>0$ one has $\overline{1}^{(4)}(t) \leq O(\epsilon)$. 
\item Fact 4: For all slots $t>T_{\epsilon}$ one has $\overline{1}^{(1)}(t) \leq O(\epsilon)$. 
\end{itemize} 

Fact 2  and Lemma \ref{lem:mu-equality} ensure that for $t > T_{\epsilon}$: 
\begin{equation} \label{eq:approx1a} 
 \overline{\mu}(t) \geq \lambda - O(\epsilon) 
 \end{equation} 
Facts 2, 3, 4 and Lemma \ref{lem:1bound} ensure that for $t > T_{\epsilon}$: 
\begin{eqnarray*}
|\overline{1}^{(2)}(t) - \theta| \leq O(\epsilon) \: \: \: , \: \: \: 
|\overline{1}^{(3)}(t) - (1-\theta)| \leq O(\epsilon)  
\end{eqnarray*}
Substituting the above into \eqref{eq:p-bound} proves that for $t>T_{\epsilon}$: 
\begin{eqnarray}
\overline{p}(t) &\leq& \theta h(\mu_{b+1}) + (1-\theta)h(\mu_b)  + O(\epsilon) \nonumber \\
&=& p^* + O(\epsilon) \label{eq:approx1b}  
\end{eqnarray}
The guarantees \eqref{eq:approx1a} and \eqref{eq:approx1b} show that the drift-plus-penalty algorithm
gives an $O(\epsilon)$-approximation with convergence time $T_{\epsilon} = O(\log(1/\epsilon)/\epsilon)$. 
This is within a factor $\log(1/\epsilon)$ of the convergence time lower bound given in Section \ref{section:converse}.
Hence, the algorithm has near-optimal convergence time.  

Further, it is known that if the rate-power curve $h(\mu)$ has at least two piecewise linear segments and if the point $(\lambda, h(\lambda))$ does not lie on the segment closest to the origin, then any algorithm 
that yields an $O(\epsilon)$-approximation must have average queue size that satisfies $\overline{Q}(t) \geq \Omega(\log(1/\epsilon))$ \cite{neely-energy-delay-it}.  Fact 1 shows that the drift-plus-penalty algorithm meets
this bound with equality.  Hence, not only does it provide near optimal convergence time, it provides an optimal average queue size tradeoff. 

\section{Practical improvements} \label{section:place-holders} 

\subsection{Place-holders} 

The structure of this problem admits a practical improvement in queue size 
via the \emph{place-holder technique} of \cite{sno-text}.  This does not change the $O(\log(1/\epsilon))$ average queue size tradeoff with $\epsilon$, 
but can reduce the \emph{coefficient} that multiplies the $\log(1/\epsilon)$ term.  Assume that $V \geq 0$ and define the following nonnegative parameter: 
\begin{equation} \label{eq:qplace}
 q_{place} \defequiv  \max\left[ \frac{V}{\omega_M} - \omega_M, 0\right] 
 \end{equation} 
The technique uses a nonzero initial condition $Q(0)=q_{place}$, where the initial backlog $q_{place}$ is \emph{fake data}, also called \emph{place-holder backlog}.   Note that $q_{place}>0$ if and only if $V>\omega_M^2$. 

The following lemma refines Lemma \ref{lem:remove-max} and shows that this place-holder backlog is never transmitted.  Hence, it acts only to shift the queue size up to a value required to make  desirable 
power allocation decisions via \eqref{eq:dpp}. 

\begin{lem} If $V \geq \omega_M^2$ and $Q(0)=q_{place}$,   then the drift-plus-penalty algorithm \eqref{eq:dpp} chooses $p(t)=\mu(t)=0$ whenever $Q(t) < V/\omega_M$.   Thus, $Q(t) \geq q_{place}$ for all $t$. 
\end{lem} 
\begin{proof} 
The proof is similar to that of Lemma \ref{lem:remove-max} and is omitted for brevity. 
\end{proof} 

Consequently, at every slot $t$ the queue can be decomposed as $Q(t) = q_{place} + Q^{real}(t)$, where $Q^{real}(t)$ is the \emph{real} queue backlog from actual arrivals.  The sample path of $Q(t)$ and all power decisions $p(t)$ are the same as when the drift-plus-penalty algorithm is implemented with the nonzero initial condition $q_{place}$.   Of course, every transmission $\mu(t)$ sends real data from the queue, rather than fake data.  The resulting algorithm is: 

\begin{itemize} 
\item Initialize $Q^{real}(0)=0$.
\item Every slot $t$, observe $Q^{real}(t)$ and $\omega(t)$ and choose: 
\[ p(t) = \left\{ \begin{array}{ll}
1 &\mbox{ if $(q_{place} + Q^{real}(t))\omega(t) \geq V$} \\
0  & \mbox{ otherwise} 
\end{array}
\right. \]
\item Update $Q^{real}(t)$ by: 
\begin{equation} \label{eq:q-real-update} 
Q^{real}(t+1) = \max[Q^{real}(t) + a(t) - p(t)\omega(t),0] 
\end{equation} 
\end{itemize} 

If $q_{place}>0$ then $q_{place} = V/\omega_M - \omega_M \leq V/\omega_b$.  Thus, $0 \leq q_{place} \leq V/\omega_b$, 
and so the initial condition $Q(0)=q_{place}$ still meets the requirements of the lemmas of the previous section. Therefore, 
the same performance
bounds hold for the power process $p(t)$ and the queue size process $Q(t)$.  However, at every instant of time, 
the \emph{real} queue size
$Q^{real}(t)$ is reduced by exactly $q_{place}$ in comparison to $Q(t)$.

\subsection{LIFO scheduling} 

The queue update equations \eqref{eq:q-real-update} and \eqref{eq:q-update} allow for any work-conserving scheduling mechanism.  
The default mechanism is First-In-First-Out (FIFO).  However, the Last-In-First-Out (LIFO) scheduling 
discipline can provide significant delay improvements for $98\%$ of the packets \cite{moeller-LIFO}\cite{longbo-LIFO-ton}. 
Intuitively, the reason is the following:  Results in the previous section show that, for sufficiently large $V$, the backlog $Q(t)$ is almost 
always to the right of the $V/\omega_{b+1}$ point in Fig. \ref{fig:drift-line}.   Suppose the place-holder technique is not used.  Then packets that arrive when $Q(t)\geq V/\omega_{b+1}$ must wait for at least $V/\omega_{b+1}$ units of data to be served under FIFO, but are transmitted more quickly under LIFO.  Work in \cite{longbo-LIFO-ton} mathematically formalizes this observation. 
Roughly speaking, most packets have average delay reduced by at least  $V/(\omega_{b+1}\lambda)$ under LIFO (and without the place-holder technique).  With the place-holder technique, this reduction is changed to 
$(V/\omega_{b+1} - q_{place})/\lambda$  (since the place-holder technique already reduces average delay of \emph{all packets} by $q_{place}/\lambda$).   One caveat is that, under LIFO,  
a finite amount of arriving data \emph{might never be transmitted}. For example, if drift-plus-penalty is implemented without the place-holder technique, then the first $q_{place}$ units of arriving data will never exit under LIFO, where $q_{place}$ is given in \eqref{eq:qplace}.   Of course, using LIFO as opposed to FIFO does not change the total queue size or the fundamental tradeoff between total average queue size and average power. 
These issues are explored via simulation in the next section.

\section{Simulation} 

\subsection{Two channel states}

\begin{figure}[htbp]
   \centering
   \includegraphics[width=3.75in]{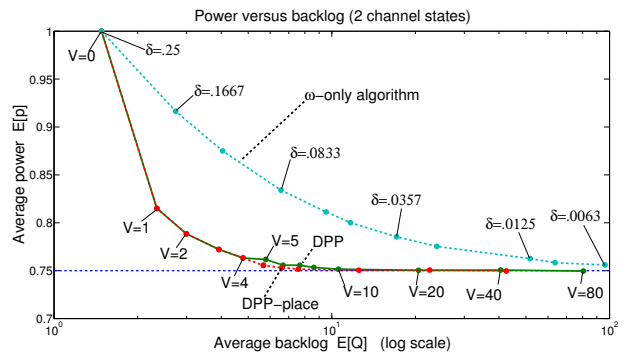} 
   \caption{Average power versus average backlog for the case of 2 channel states.  All data points are averages obtained after simulation over 1 million slots.  Three algorithms are shown.  The drift-plus-penalty (DPP) algorithms use various values of $V$.  The $V$ values are labeled for select points on the DPP curve (green).  The $\omega$-only algorithm uses various values of $\delta$. The DPP algorithms significantly outperform $\omega$-only.  DPP-place (red) provides only a modest gain over DPP (green) in the range $V \in [4, 10]$.}
   \label{fig:2-channels}
\end{figure}

Consider the scenario of Case 1 in Section \ref{section:converse}. There are two channel states $\omega(t) \in \{1, 2\}$ 
with $\pi(1) = 3/4, \pi(2) = 1/4$.  The $h(\mu)$ curve is shown in Fig. \ref{fig:case1}. Assume the arrival process $a(t)$ is i.i.d. over slots with: 
\[ Pr[a(t)=0]=\frac{2}{5} \: , \: Pr[a(t)=1] = \frac{1}{5} \: , \: Pr[a(t)=2] = \frac{2}{5} \]
The arrival rate is $\lambda = \expect{a(t)} = 1$, and the minimum average power required for stability is $p^*=h(1) = 3/4$. 

Three different algorithms are considered below: 
\begin{itemize} 
\item Drift-plus-penalty (DPP) with $Q(0)=0$. 
\item DPP with place-holder (DPP-place) with $q_{place} = \max[V/2-2,0]$ (from \eqref{eq:qplace}) and $Q^{real}(0)=0$.  
\item An $\omega$-only policy designed to satisfy $\expect{\mu(t)} = \lambda + \delta$ and $\expect{p(t)} = h(\lambda + \delta)$. 
\end{itemize} 
The DPP algorithms operate online without knowledge of $\lambda$, $\pi(1)$, $\pi(2)$, while the $\omega$-only policy is designed offline with knowledge of these values.  Results are plotted in Fig. \ref{fig:2-channels} for various values of $V\geq 0$ and $\delta\geq 0$. 
\emph{The DPP  algorithms significantly outperform the $\omega$-only algorithm even though they do not have knowledge of the system probabilities}. 
The theoretical tradeoffs of the previous section were derived under the assumption that $V \geq \omega_M^2$ (in this case, $\omega_M^2 = 2^2=4)$.  However, the DPP algorithms can be implemented for any value $V \geq 0$.  Observe from the figure that 
average power starts approaching optimality even for values $V<4$, and converges to the optimal $p^*=3/4$ as $V$ is increased beyond 4.  It can be shown that the $\omega$-only algorithm achieves an $O(\epsilon)$-approximation with average queue size $\Theta(1/\epsilon)$, whereas results in the previous section prove the DPP algorithms achieve an $O(\epsilon)$-approximation with average queue size $\Theta(\log(1/\epsilon))$. 
The simulations verify these theoretical results.

In this example, the DPP place-holder algorithm gives performance very close to standard DPP, with only a modest gain in the range $V\in[4, 10]$.  For values $V \leq 4$ the DPP and DPP-place algorithms are identical. 

Convergence time to the desired constraint $\overline{\mu}(t) \geq \lambda$ is illustrated in Fig. \ref{fig:mu-time} by plotting the empirical value of $\expect{\mu(t)}$ versus time. 
  The $\omega$-only policy is not plotted because it achieves the constraint immediately by its offline design.   
  The DPP-place algorithm shows a slight convergence time improvement over DPP. 
Both DPP algorithms demonstrate that $|\overline{\mu}(t) - \lambda|$ decays like $V/t$.   This is consistent with the theoretical guarantees derived in the previous section.  Indeed, for an $O(\epsilon)$-approximation, one sets $V = \Theta(\log(1/\epsilon))$, so after time $t \geq \Theta(\log(1/\epsilon)/\epsilon)$ the 
deviation from the constraint is at most $O(V/t) \leq O(\epsilon)$.  The corresponding average power $\expect{p(t)}$ is 
plotted in Fig. \ref{fig:p-time}.

\begin{figure}[htbp]
   \centering
   \includegraphics[width=3.75in]{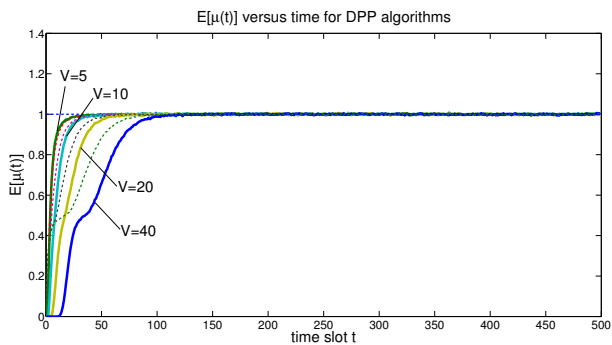} 
   \caption{Average transmission rate $\expect{\mu(t)}$ versus time, obtained from $10^5$ independent simulation runs over the first 500 slots.  The DPP curves are thick, solid, and labeled with $V \in \{5, 10, 20, 40\}$.  The DPP-place curves are thin dashed curves where $V \in \{5, 10, 20, 40\}$ corresponds to red, purple, grey, green, respectively.}
   \label{fig:mu-time}
\end{figure}

\begin{figure}[htbp]
   \centering
   \includegraphics[width=3.75in]{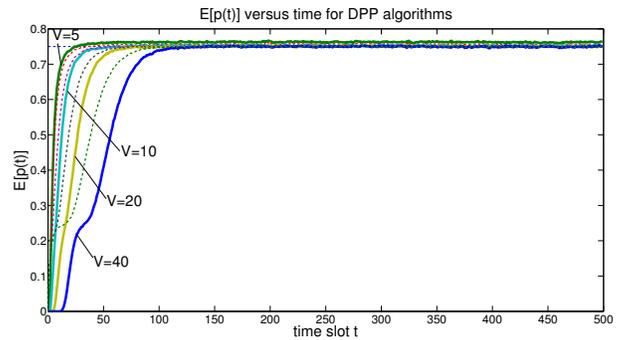} 
   \caption{Average power $\expect{p(t)}$ versus time for the same experiments, $V$ parameters, and color scheme 
   as Fig. \ref{fig:mu-time}.}
   \label{fig:p-time}
\end{figure}

\subsection{Nine channel states} \label{section:nine-channels} 

Now consider a process $\omega(t)$ with 9 possible rates $\{\omega_0, \ldots, \omega_9\}$: 
\[ \Omega = \{0, 3, 7, 11, 18, 22, 24, 36, 46\} \]
The probabilities are: 
\begin{eqnarray*}
\pi(\omega_i) &=&  \left\{ \begin{array}{ll}
1/15 &\mbox{ if $i \in \{0, 1, 2\}$} \\
2/9  & \mbox{ if $i \in \{3, 4, 5\}$} \\
2/45 & \mbox{ if $i \in \{6, 7, 8\}$} 
\end{array}
\right.
\end{eqnarray*}
The arrival process $a(t)$ has probabilities: 
\[ Pr[a(t)=0] = 0.42, Pr[a(t)=20] = 0.58 \]
with arrival rate $\lambda = 11.6$ packets/slot.  The DPP-place 
algorithm uses $q_{place} = \max[V/46 - 46, 0]$ (as in \eqref{eq:qplace}), and $q_{place}>0$ if and only if $V>46^2=2116$.   It can be shown that $p^* = h(\lambda) = 7/15$ for this system.   Simulations for DPP and DPP-place are in Fig. \ref{fig:9-channels}.  
As before, the DPP algorithms outperform the $\omega$-only policy, although the improvements are not as dramatic as they are in Fig. \ref{fig:2-channels}.  This is because the arrival rate vector in this case  is close to a vertex point 
of the $h(\mu)$ curve.  As before, the 
DPP-place algorithm performance is similar to that of DPP with a shifted $V$ parameter. 

\begin{figure}[htbp]
   \centering
   \includegraphics[width=3.75in]{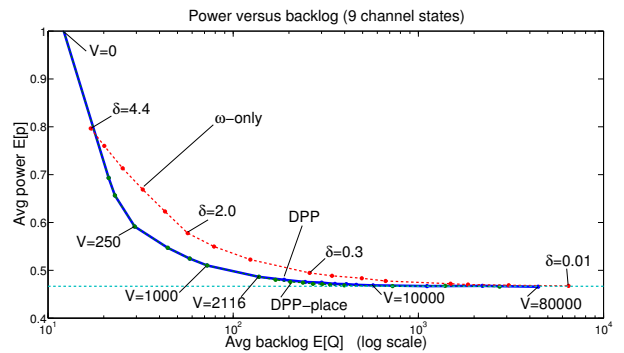} 
   \caption{The 9-channel simulation}
   \label{fig:9-channels}
\end{figure}

\subsection{Robustness to non-ergodic changes} 

This subsection illustrates how the algorithm reacts to nonergodic changes.  The system with 9 possible channel states from the previous section is considered.  The simulation is run over 6000 slots, broken into three phases of 2000 slots each.  The system probabilities are changed at the beginning of each phase.  The algorithm is not aware of the changes and must adapt.  Specifically:

\begin{enumerate} 
\item First phase:  The same parameters of the previous subsection are used (so  $\lambda = 11.6$).   

\item Second phase: Channel probabilities are the same as phase 1. The arrival rate is increased
to $\lambda = 13$ by using $Pr[a(t)=20] = .65$, $Pr[a(t)=0] =  0.35$. 

\item Third phase:  The same arrival rate $\lambda = 13$ of phase 2 is used. However, channel probabilities are changed to: 
\begin{eqnarray*}
\pi(\omega_i) &=&  \left\{ \begin{array}{ll}
1/15 &\mbox{ if $i \in \{0, 1, 2\}$} \\
1/9  & \mbox{ if $i \in \{3, 4, 5\}$} \\
7/45 & \mbox{ if $i \in \{6, 7, 8\}$} 
\end{array}
\right.
\end{eqnarray*}
\end{enumerate} 

\begin{figure}[htbp]
   \centering
   \includegraphics[width=3.75in]{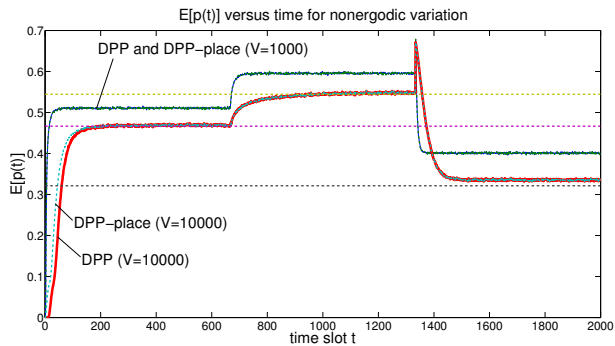} 
   \caption{Average power (obtained from 10000 independent simulation runs) versus time for a system that changes nonergodically over 3 phases.}
   \label{fig:nonerg1}
\end{figure}

\begin{figure}[htbp]
   \centering
   \includegraphics[width=3.75in]{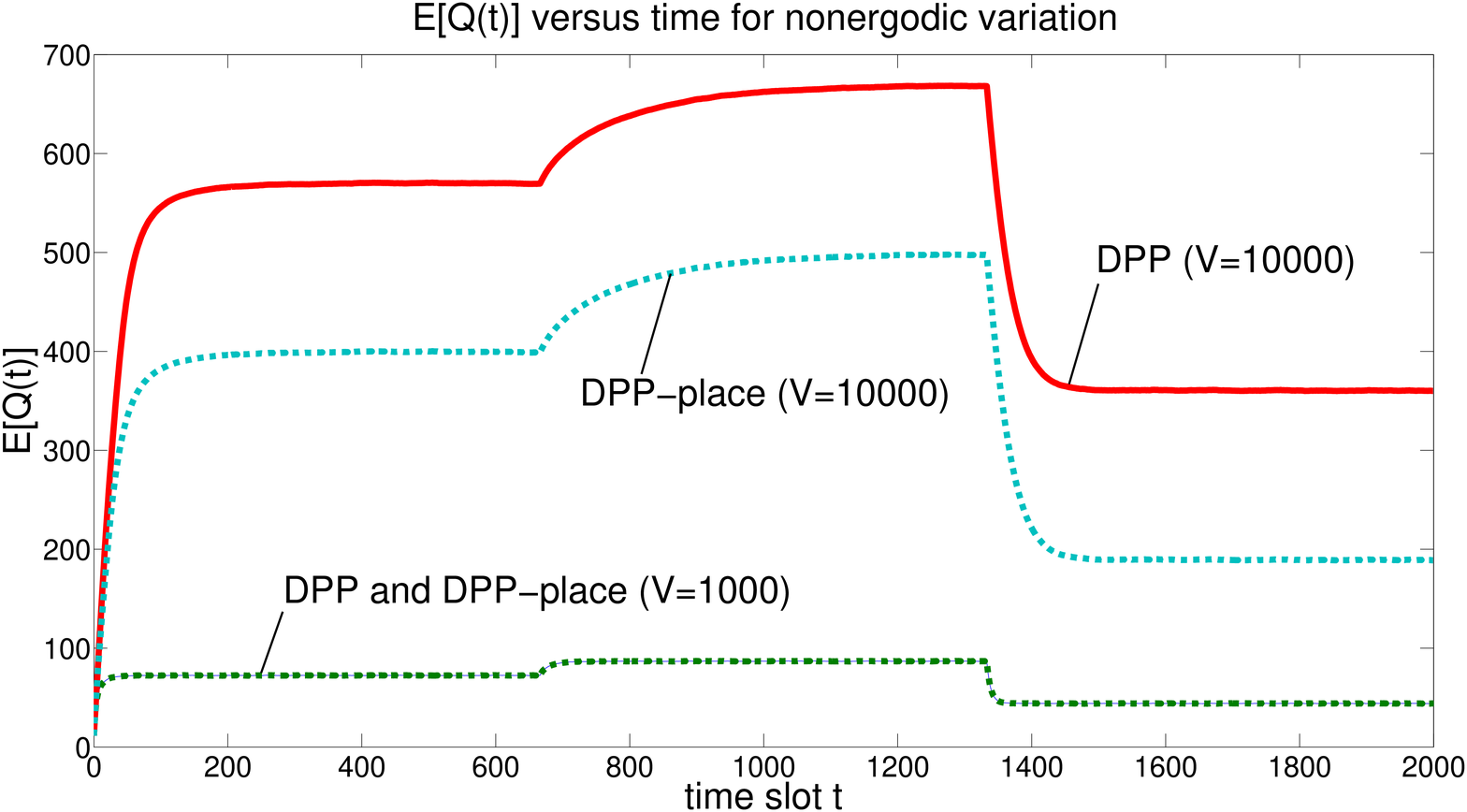} 
   \caption{Average queue size (obtained from 10000 independent simulation runs) versus time for a system that changes nonergodically over 3 phases.}
   \label{fig:nonerg2}
\end{figure}

The resulting power and queue size averages 
are plotted in Figs. \ref{fig:nonerg1} and \ref{fig:nonerg2}. 
The data is obtained by averaging sample paths over $10000$ independent runs.  Fig. \ref{fig:nonerg1} shows that for large $V$, average power  converges to a value close to 
the long-term optimum associated with each phase.  Thus, the DPP algorithms adapt to changing environments.   
For each $V$, average power of DPP-place is roughly the same  as DPP (Fig. \ref{fig:nonerg1}).  Average queue size of DPP-place is smaller than that of DPP when $V$ is large (Fig. \ref{fig:nonerg2}).

\subsection{Delay improvements under LIFO}

Fig. \ref{fig:DPPLIFO} illustrates the gains of Last-in-First-Out (LIFO) scheduling (as in \cite{moeller-LIFO}\cite{longbo-LIFO-ton}) for the 9-channel state system with parameters described in Section \ref{section:nine-channels} (the system is the same as that of Fig. \ref{fig:9-channels}).  Average power is plotted versus average delay (in slots) for DPP-place with and without LIFO.  The LIFO data considers only the $98\%$ of all packets with the smallest delay  (so that $2\%$ of the packets are ignored in the delay computation).  LIFO scheduling significantly reduces delay for these packets.  For example, when $V=80000$, average delay is 236.3 slots without LIFO, and only 20.0 slots with LIFO (average power is the same for both algorithms).

\begin{figure}[htbp]
   \centering
   \includegraphics[width=3.75in]{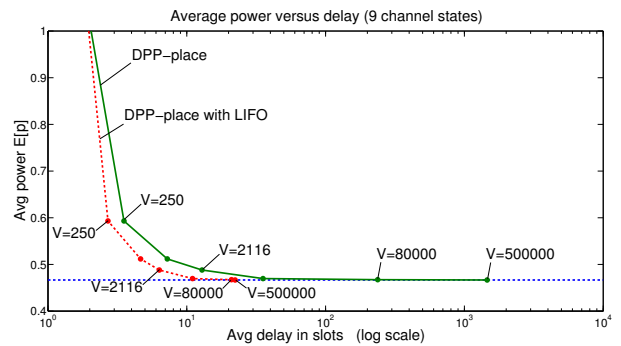} 
   \caption{A comparison of the power-delay tradeoff for DPP-place with and without LIFO. The LIFO data for average delay considers  the best $98 \%$ of all traffic.  Each data point represents a simulation over $10^6$ slots. Average power is the same for both algorithms whenever $V$ is the same.}
   \label{fig:DPPLIFO}
\end{figure}

\section{Conclusions} 

This paper considers convergence time for minimizing average power in a 
wireless transmission link with time varying channels and random traffic.  Prior algorithms produce
an $\epsilon$-approximation with convergence time $O(1/\epsilon^2)$.  This paper shows, for a simple
example, that no algorithm can get convergence time better than $O(1/\epsilon)$.  It then shows that this ideal convergence time tradeoff 
can be approached to within a logarithmic factor. 
Furthermore, the resulting average queue size is at most $O(\log(1/\epsilon))$, which is known 
to be an optimal tradeoff.  This establishes fundamental convergence time, queue size, and power characteristics of 
wireless links.  It shows that learning times in an unknown environment 
can be pushed much faster than expected. 

\bibliographystyle{unsrt}
\bibliography{../../../latex-mit/bibliography/refs}

\begin{thebibliography}{10}

\bibitem{tass-server-allocation}
L.~Tassiulas and A.~Ephremides.
\newblock Dynamic server allocation to parallel queues with randomly varying
  connectivity.
\newblock {\em IEEE Transactions on Information Theory}, vol. 39, no. 2, pp.
  466-478, March 1993.

\bibitem{atilla-fairness-ton}
A.~Eryilmaz and R.~Srikant.
\newblock Fair resource allocation in wireless networks using
  queue-length-based scheduling and congestion control.
\newblock {\em IEEE/ACM Transactions on Networking}, vol. 15, no. 6, pp.
  1333-1344, Dec. 2007.

\bibitem{atilla-primal-dual-jsac}
A.~Eryilmaz and R.~Srikant.
\newblock Joint congestion control, routing, and {MAC} for stability and
  fairness in wireless networks.
\newblock {\em IEEE Journal on Selected Areas in Communications, Special Issue
  on Nonlinear Optimization of Communication Systems}, vol. 14, pp. 1514-1524,
  Aug. 2006.

\bibitem{lee-stochastic-scheduling}
J.~W. Lee, R.~R. Mazumdar, and N.~B. Shroff.
\newblock Opportunistic power scheduling for dynamic multiserver wireless
  systems.
\newblock {\em IEEE Transactions on Wireless Communications}, vol. 5, no.6, pp.
  1506-1515, June 2006.

\bibitem{prop-fair-down}
H.~Kushner and P.~Whiting.
\newblock Asymptotic properties of proportional-fair sharing algorithms.
\newblock {\em Proc. 40th Annual Allerton Conf. on Communication, Control, and
  Computing, Monticello, IL}, Oct. 2002.

\bibitem{vijay-allerton02}
R.~Agrawal and V.~Subramanian.
\newblock Optimality of certain channel aware scheduling policies.
\newblock {\em Proc. 40th Annual Allerton Conf. on Communication, Control, and
  Computing, Monticello, IL}, Oct. 2002.

\bibitem{stolyar-greedy}
A.~Stolyar.
\newblock Maximizing queueing network utility subject to stability: Greedy
  primal-dual algorithm.
\newblock {\em Queueing Systems}, vol. 50, no. 4, pp. 401-457, 2005.

\bibitem{neely-fairness-ton}
M.~J. Neely, E.~Modiano, and C.~Li.
\newblock Fairness and optimal stochastic control for heterogeneous networks.
\newblock {\em IEEE/ACM Transactions on Networking}, vol. 16, no. 2, pp.
  396-409, April 2008.

\bibitem{sno-text}
M.~J. Neely.
\newblock {\em Stochastic Network Optimization with Application to
  Communication and Queueing Systems}.
\newblock Morgan \& Claypool, 2010.

\bibitem{shroff-opportunistic}
X.~Liu, E.~K.~P. Chong, and N.~B. Shroff.
\newblock A framework for opportunistic scheduling in wireless networks.
\newblock {\em Computer Networks}, vol. 41, no. 4, pp. 451-474, March 2003.

\bibitem{neely-energy-it}
M.~J. Neely.
\newblock Energy optimal control for time varying wireless networks.
\newblock {\em IEEE Transactions on Information Theory}, vol. 52, no. 7, pp.
  2915-2934, July 2006.

\bibitem{berry-fading-delay}
R.~Berry and R.~Gallager.
\newblock Communication over fading channels with delay constraints.
\newblock {\em IEEE Transactions on Information Theory}, vol. 48, no. 5, pp.
  1135-1149, May 2002.

\bibitem{neely-energy-delay-it}
M.~J. Neely.
\newblock Optimal energy and delay tradeoffs for multi-user wireless downlinks.
\newblock {\em IEEE Transactions on Information Theory}, vol. 53, no. 9, pp.
  3095-3113, Sept. 2007.

\bibitem{longbo-LIFO-ton}
L.~Huang, S.~Moeller, M.~J. Neely, and B.~Krishnamachari.
\newblock {LIFO}-backpressure achieves near optimal utility-delay tradeoff.
\newblock {\em IEEE/ACM Transactions on Networking}, vol. 21, no. 3, pp.
  831-844, June 2013.

\bibitem{longbo-lagrange-tac}
L.~Huang and M.~J. Neely.
\newblock Delay reduction via {L}agrange multipliers in stochastic network
  optimization.
\newblock {\em IEEE Transactions on Automatic Control}, vol. 56, no. 4, pp.
  842-857, April 2011.

\bibitem{atilla-convergence-infocom2013}
B.~Li and A.~Eryilmaz.
\newblock Wireless scheduling for network utility maximization with optimal
  convergence speed.
\newblock In {\em Proc. IEEE INFOCOM}, Turin, Italy, April 2013.

\bibitem{dist-opt-arxiv}
M.~J. Neely.
\newblock Distributed stochastic optimization via correlated scheduling.
\newblock {\em ArXiv technical report, arXiv:1304.7727v2}, May 2013.

\bibitem{martingale-concentration}
F.~Chung and L.~Lu.
\newblock Concentration inequalities and martingale inequalities--a survey.
\newblock {\em Internet Mathematics}, vol. 3, pp. 79-127, 2006.

\bibitem{moeller-LIFO}
S.~Moeller, A.~Sridharan, B.~Krishnamachari, and O.~Gnawali.
\newblock Routing without routes: The backpressure collection protocol.
\newblock {\em Proc. 9th ACM/IEEE Intl. Conf. on Information Processing in
  Sensor Networks (IPSN)}, April 2010.

\end{thebibliography}
\end{document}